\documentclass[11pt]{article}
\usepackage{amsmath}
\usepackage{amssymb}
\usepackage{theorem}

\let\optbf\bf 
\let\optbf\null

\usepackage{enumerate}

\usepackage{mathrsfs}

\usepackage{geometry}

\usepackage{hyperref}

\numberwithin{equation}{section}

\newtheorem{thm}{Theorem}[section]
\newtheorem{lemma}[thm]{Lemma}
\newtheorem{prop}[thm]{Proposition}
\newtheorem{cor}[thm]{Corollary}
{\theorembodyfont{\rmfamily}
\newtheorem{defn}[thm]{Definition}

\newtheorem{rmk}[thm]{Remark}
\newtheorem{notation}[thm]{Notation}
}

\newcommand{\qed}{\hfill \mbox{\raggedright \rule{.07in}{.1in}}}

\newenvironment{proof}{\vspace{1ex}\noindent{\bf
Proof.}\hspace{0.5em}}{\hfill\qed\vspace{1ex}}
\newenvironment{pfof}[1]{\vspace{1ex}\noindent{\bf Proof of
#1.}\hspace{0.5em}}{\hfill\qed\vspace{1ex}}

\newcommand{\alphaupperbound}{\frac{1}{9}}

\newcommand{\bbb}{\beta} 

\newcommand{\one}{\ensuremath{\mathbf 1}}

\renewcommand{\H}{{\bold{H}}}

\let\phi\varphi

\let\epsilon\varepsilon

\let\tilde\widetilde

\newcommand{\cone}{\ensuremath{{\mathcal C}_2}}

\newcommand{\eval}[1]{\downharpoonright_{#1}}

\newcommand{\Var}{\operatorname{Var}}

\newcommand{\mean}[2]{\left[#2\right]_{#1}}

\newcommand{\Sprindzuk}{Sprindzuk's }

\newcommand{\R}{{\mathbb R}}

\newcommand{\N}{{\mathbb N}}

\newcommand{\Lip}{\operatorname{Lip}}

\def\a{\alpha}
\def\b{\beta}

\def\R{\ensuremath{\mathbb R}}
\def\N{\ensuremath{\mathbb N}}

\def\B{\ensuremath{\mathcal B}}

\def\P{\ensuremath{\mathcal P}}

\def\E{\mathbb E}

\def\TF{\mathcal{T}}

\begin{document}

\bibliographystyle{plain}

\title {Central limit theorems for sequential and random  intermittent dynamical systems.}

\author{Matthew Nicol \thanks{Department of Mathematics,
University of Houston,
Houston Texas,
USA. e-mail: $<$nicol@math.uh.edu$>$.}
\and Andrew T\"or\"ok \thanks{Department of Mathematics,
University of Houston,
Houston Texas,
USA
 and
 Institute of Mathematics of
 the Romanian Academy, P.O. Box 1--764, RO-70700
 Bucharest, Romania.
 e-mail: $<$torok@math.uh.edu$>$.}
\and Sandro Vaienti
\thanks{Aix Marseille Universit\'e, CNRS, CPT, UMR 7332, 13288 Marseille, France and
Universit\'e de Toulon, CNRS, CPT, UMR 7332, 83957 La Garde, France.
e-mail:$<$vaienti@cpt.univ-mrs.fr$>$.}}

\date{June 27, 2016}


\maketitle


\tableofcontents{}

\begin{abstract}
  We establish self-norming central limit theorems for non-stationary
  time series arising as observations on sequential maps possessing an
  indifferent fixed point. These transformations are obtained by
  perturbing the slope in the Pomeau-Manneville map. We also obtain
  quenched central limit theorems for random compositions of these
  maps.
\end{abstract}

\section{Introduction}

In a preceding series of two papers \cite{HNTV}, \cite{lom-paper}, we
considered a few statistical properties of non-stationary dynamical systems
arising by the sequential composition of (possibly) different maps. The
first article \cite{HNTV} dealt with the Almost Sure Invariance Principle
(ASIP) for the non-stationary process given by the observation along the
orbit obtained by concatenating maps chosen in a given set. We choose maps
in one and more dimensions which were piecewise expanding, more precisely
their transfer operator (Perron-Frobenius, "PF") with respect to the
Lebesgue measure was quasi-compact on a suitable Banach space.
This allows to approximate the original process with a reverse martingale
plus an error. By a recent result by Cuny and Merlev\`ede
\cite{Cuny_Merlevede}, the reverse martingale satisfies the ASIP. The error
is shown to be essentially bounded due to the presence of a spectral gap in
the PF operator on a Banach space continuously injected in $L^{\infty}$
(from now on all the $L^p$ spaces will be with respect to the ambient
Lebesgue measure $m$ and they will be denoted with $L^p$ or $L^p(m).$).
Moreover, the same spectral property allowed us to show that for expanding
maps chosen close enough, the variance $\sigma_n^2$ grows linearly, which
permits to approximate the original process almost everywhere with a finite
sum of i.i.d. Gaussian variables with the same variance.

The second paper \cite{lom-paper} considered composition of
Pomeau-Manneville like maps, obtained by perturbing the slope at the
indifferent fixed point $0.$ We got polynomial decay of correlations for
particular classes of centered observables, which could also be interpreted
as the decay of the iterates of the PF operator on functions of zero
(Lebesgue) average; this fact is also known as {\em loss of memory.} In
this situation the PF operator is not quasi-compact and although the
process given by the observation along a sequential orbit can be decomposed
again as the sum of a reverse martingale difference plus an error, apriori
the latter turns out to be bounded only in $L^1$ and this was an obstacle
to obtain an almost sure result like the ASIP by {\em only} looking at the
almost sure convergence of the reverse martingale difference. Instead one
could hope to get a (distributional) Central Limit Theorem (CLT); in this
regard a general approach to CLT for sequential dynamical systems has been
proposed and developed in \cite{Conze_Raugi}. It basically applies to
systems with a quasi-compact PF operator and it is not immediately
transposable to maps with do not admit a spectral gap. The main goal of our
paper is to prove the CLT for the sequential composition of
Pomeau-Manneville maps with varying slopes. A fundamental tool in obtaining
such a result will be the polynomial loss of memory bound obtained
in~\cite{lom-paper}; we are now going to recall it also because it will
determine the regularity of the observables to which our CLT will apply;
see Theorem~\ref{thm:decay}.

We consider the family of Pomeau-Manneville maps
\begin{equation}\label{PM} T_{\a}(x)= \begin{cases} x + 2^{\a}x^{1+\a}, \
0\le x \le 1/2\\ 2x-1, \ 1/2 \le x \le 1
  \end{cases} \qquad 0 < \a <1.
\end{equation} Actually in \cite{lom-paper} we considered a slightly
different family of this type, but pointed out that both versions could be
worked out with the same techniques (see~\cite{Aimino_thesis}), and lead to
the same result; here we prefer to use the {\em classical} version
(\ref{PM}). As in \cite{LSV}, we identify the unit interval $[0,1]$ with
the circle $S^1$, so that the maps become continuous. Given $0< \b_k \le
\a <1$, denote by $P_{\b_k}$ or $P_k$ the {\em Perron-Frobenius} operator
associated with the map $T_k= T_{\b_k}$ w.r.t. the measure $m$. For
concatenations we use equivalently the notations
\[ \TF_m^{n-m+1}:=T_{\b_{n}}\circ T_{\b_{n-1}}\circ \cdots \circ T_{\b_m}=
T_{n}\circ T_{n-1}\circ\cdots \circ T_m.
\]
\[ \P_m^{n-m+1}:=P_{\b_{n}}\circ P_{\b_{n-1}}\circ \cdots \circ P_{\b_m}=
P_{n}\circ P_{n-1}\circ\cdots \circ P_m.
\]
\[
\P^{n}:=\P_1^{n} \qquad \TF^{n}:=\TF_1^{n}
\]
where \emph{the exponent denotes the number of maps in the concatenation}.
For simplicity we use $\TF^{\infty}:= \cdots T_n\circ \cdots \circ T_1$ for
a given sequence of transformations.

The Perron-Frobenius operator $P_{k}$ associated to $T_{k}$ satisfies the
duality relation
\[
\int_M P_{k}f \ g\ dm\ = \ \int_M f\ g\circ T_{k} \ dm,\;\; \mbox{ for
  all } f\in L^1, \ g\in L^{\infty}
\]
and this is preserved under concatenation.

We next consider~\cite{LSV, lom-paper} the cone {\cone} of functions given
by (here $X(x)=x$ is the identity function):
\[
\cone:= \{f\in C^0((0,1])\cap L^1(m) \mid \ f\ge 0, \ f \
\mbox{decreasing},\ X^{\a+1}f \ \mbox{increasing}, \ f(x)\le a x^{-\a}\
m(f)\}\footnote{By "decreasing" we mean "nonincreasing".}
\]

\begin{rmk}
  Some coefficients that appear later depend on the value $a$ that defines
  the cone $\cone$; however, we will not write explicitly this dependence.
\end{rmk}

Fix $0< \a < 1$; as proven in~\cite{lom-paper}, provided $a$ is large
enough, the cone {\cone} is preserved by all operators
$P_{\b}, \ 0< \b \le \a < 1$. The following polynomial decay result holds:

\begin{thm}[\cite{lom-paper}]\label{thm:decay} Fix $0<\a<1$ and consider
  a cone $\cone$ as above. Suppose $\psi, \phi$ in $\cone$ have equal
  expectation, $\int \phi dm= \int \psi dm$. Then for any sequence
  $T_{\b_1},\cdots, T_{\b_n}$, $n \ge 1$, of maps of Pomeau-Manneville type
  \eqref{PM} with $0 < \b_k\le \a < 1$, $k\in [1,n]$, we have
  \begin{equation}\label{DC}
    \int |P_{\b_n}\circ\cdots\circ
    P_{\b_1}(\phi)-P_{\b_n}\circ\cdots\circ P_{\b_1}(\psi)|dm \le C_{\a}
    (\|\phi\|_1+\|\psi\|_1)n^{-\frac{1}{\a}+1}(\log n)^{\frac{1}{\a}},
  \end{equation}
  where the constant $C_{\a}$ depends only on the map $T_{\a}$, and
  $\|\cdot\|_1$ denotes the $L^1$ norm.

  A similar rate of decay holds for observables $\phi$ and $\psi$ that are
  $C^1$ on $[0,1]$; in this case the rate of decay has an upper bound given
  by
  \[ C_{\a} \
  \mathcal{F}(\|\phi\|_{C^1}+\|\psi\|_{C^1})n^{-\frac{1}{\a}+1}(\log
  n)^{\frac{1}{\a}}
  \] where the function $\mathcal{F}: \R\rightarrow \R$ is affine.
\end{thm}

For the proof of the CLT Theorem~\ref{thm:CLT} we need better decay than in
$L^1$. In this paper we improve the above result to decay in $L^p$,
provided $\a$ is small enough.


Note that $\P^n\phi \in \cone$ if $\phi\in \cone$ and
$m(\P^n\phi)=m(\phi)$, so
\[
\left|[\P^n(\phi) - \P^n(\psi)]\eval{x} \right| \le
\left|\P^n(\phi)\eval{x} \right| + \left|\P^n(\psi)\eval{x} \right|
\le a m(\phi) x^{-\alpha} + a m(\psi) x^{-\alpha}
\]

\begin{prop}\label{prop:decay_Lp}

  Under the assumptions on Theorem~\ref{thm:decay}, if $1 \le p <
  1/\a$ then
  \begin{equation}\label{DC_Lp}
    \|P_{\b_n}\circ\cdots\circ
    P_{\b_1}(\phi)-P_{\b_n}\circ\cdots\circ P_{\b_1}(\psi)\|_{L^p(m)}\le
    C_{\a,p} (\|\phi\|_1+\|\psi\|_1)n^{1 -\frac{1}{p \alpha }}
    \left(\log n\right)^{\frac{1}{\alpha}\frac{1-\a p}{p- \a p}}
  \end{equation}
  where the constant $C_{\a,p}$ depends only on the map $T_{\a}$ and
  $p$.

  As in Theorem~\ref{thm:decay}, a similar $L^p$-decay result also holds
  for observables $\phi, \psi \in C^1([0,1])$.
\end{prop}

\begin{proof}
  For functions in the cone $\cone$, Theorem~\ref{thm:decay} gives
  $L^1$-decay; then Lemma~\ref{lem:L2bounds} together with the
  preceding discussion implies $L^{p}$-decay for $\alpha$ small
  enough. Note that we use this Lemma with
  $K=2a(\|\phi\|_1+\|\psi\|_1)$ and the $L^1$-bound given by the
  Theorem, and then the coefficient in the $L^p$-bound is proportional
  to $(\|\phi\|_1+\|\psi\|_1)$ as well.

  To prove the decay for $C^1$ observables, we use Lemma~\ref{lem:C1*Pk}
  (same approach as in the proof of Theorem~\ref{thm:decay}).
\end{proof}

Note that the convergence of the quantity (\ref{DC}) implies the decay of
the non-stationary correlations with respect to $m$:
\[
\begin{aligned} \left|\int \psi \phi\circ T_{\b_n}\circ \cdots \circ
T_{\b_1} dm-\int \psi dm \ \int \phi\circ T_{\b_n}\circ \cdots \circ
T_{\b_1} dm \right| \\ \le \|\phi\|_{\infty} \left\|P_{\b_n}\circ \cdots
\circ P_{\b_1}(\psi)-P_{\b_n}\circ \cdots \circ P_{\b_1} \left(\one
\left(\int \psi dm\right)\right) \right\|_1 \end{aligned}
\] provided $\phi$ is essentially bounded and $(\int \psi dm)\one$ is
in the functional space where the convergence of (\ref{DC}) takes
place. In particular, this holds for $C^1$ observables, by
Theorem~\ref{thm:decay}.


As it is suggested by the preceding loss of memory result, centering the
observable is the good way to define the process when it is not stationary,
in order to consider limit theorems. To simplify the exposition, we
introduce the following notation:

\begin{defn}\label{def.mean}
  For $\phi:[0,1]\to \R$ sufficiently regular (often $C^1$) introduce the
  following normalization {\em along a sequential orbit}:
  \begin{equation}\label{eq.mean}
    \mean{k}{\phi}:=\phi-\int \phi( T_k \circ \cdots \circ T_1)
    dm.
  \end{equation}
  However, to simplify notation, it is convenient to set
  $\mean{0}{\phi}=0$.
\end{defn}

Conze and Raugi \cite{Conze_Raugi}
defined the sequence of transformations $\{T_1, T_2, T_3, \dots\}$ to be
{\em pointwise ergodic} whenever the law of large numbers is satisfied,
namely
\[
  \lim_{n\rightarrow \infty}\frac1n \sum_{k=1}^n\left[\phi( T_k \circ
    \cdots \circ T_1 x)-\int \phi( T_k \circ \cdots \circ T_1) dm\right]=0
  \text{ for Lebesgue-a.e. $x$.}
\]
We will prove in Theorem \ref{thm:SBC} that such a law of large numbers
holds for our observations provided $0<\alpha<1.$ It is therefore natural
to ask about a non-stationary Central Limit Theorem for the sums
\begin{equation}\label{eq:defn-S_n}
  S_n:=\sum_{k=1}^{n}\mean{k}{\phi} \circ T_{k}\circ \cdots \circ T_{1}
\end{equation}
for a given sequence
$\TF^{\infty}:= \cdots \circ T_n\circ \cdots \circ T_1:$ this will be the
content of the next sections.


To be more specific we will prove in Theorem~\ref{thm:CLT} a non-stationary
central limit theorem similar to that proved by Conze and
Raugi~\cite{Conze_Raugi} for (piecewise expanding) sequential systems:
\begin{equation}\label{cltt}
\frac{S_n}{\sqrt{\Var(S_n)}} \to^{d} \mathcal{N}(0,1).
\end{equation}

At this point, we would like to make a few comments about our result
compared to that of Conze and Raugi. Theorem 5.1 in \cite{Conze_Raugi}
shows that, when applied to the quantities defined above and for
classes of maps enjoying a quasi-compact transfer operator:
\begin{enumerate}
\item[(1)] If the norms $||S_n||_2$ are bounded, then the sequence
  $S_n, n\ge 1$ is bounded.
\item[(2)] If $||S_n||_2\rightarrow \infty$, then \eqref{cltt} holds.
\end{enumerate}
We are not able to prove item (1) for the intermittent map following
the same approach as in \cite{Conze_Raugi}, since it uses the uniform
boundedness of the sequence $\H_n\circ \mathcal{T}^k$, where the
function $\H_n$ is defined in (\ref{eq:h_n}) and is just the error in
the martingale approximation as we discussed above. We can only prove
that $\H_n$ is bounded uniformly in $n$ on each set of the form $[a,
1), a>0$, and do not expect it to be bounded near 0 (look at the
stationary case).

Instead, our central limit theorem will satisfy item (2) under the
assumption that the variance $||S_n||_2^2$ grows at a certain rate and for some limitation on the range of values of $\alpha$.  It
seems difficult to get such a result in full generality for the
intermittent map considered here. Conze and Raugi proved the linear
growth of the variance in their Theorem 5.3 under a certain number of
assumptions, including the presence of a spectral gap for the transfer
operator. We showed in our paper \cite{HNTV} that those assumptions
apply to several classes of expanding maps even in higher
dimensions.

However, for concatenations given by the {\em same} intermittent map $T_\a$
with $\alpha<1/2,$ the variance is linear in $n$, provided the observable
is not a coboundary for $T_\a$. In section~\ref{sec:nearby-maps} we prove
that the linear growth of the variance still holds if we take maps
$T_{\beta_n}$ with $\beta_n$ arbitrary but close to a fixed $\bbb$, and an
observable is not a coboundary for $T_\bbb$; therefore, the CLT holds. See
Theorem~\ref{thm:linear-variance}. Our proof of
Theorem~\ref{thm:linear-variance} uses an estimate of interesting related
work of Lepp{\"a}nen and Stenlund~\cite{Leppanen-Stenlund}, which we learnt
about after a first version of this paper was completed. Their result
allowed us to give another example where variance grows linearly for a
sequential dynamical system of intermittent type maps, and hence the
non-stationary CLT holds. The focus of~\cite{Leppanen-Stenlund} is however
more on the strong law of large numbers and convergence in probability
rather than the CLT. They also consider quasi static systems, introduced
in~\cite{Leppanen-Stenlund2}.

In section~\ref{sec:quenched} we show that the variance grows linearly for
almost all sequences when we compose intermittent maps chosen from a finite
set and we take them according to a fixed probability distribution. This
means that for almost all sequences (with respect to the induced Bernoulli
measure) of maps, the central limit theorem holds (a {\em quenched} CLT).
See Theorem~\ref{thm:CLTquenched}.

\begin{rmk}\label{rmk:no-log-factor}
  For simplicity, in many of the following statements we will use as rate
  of decay $n^{-\frac{1}{\a}+1}$, ignoring the $\log n$-factor. This is
  correct if we take for $\alpha$ a slightly larger value (and is actually
  the correct rate of decay for the stationary case).
\end{rmk}

\begin{notation}
  For any sequences of numbers $\{a_n\}$ and $\{b_n\}$, we will write
  $a_n\approx b_n$ if $c_1b_n\le a_n\le c_2b_n$ for some constants
  $c_2\ge c_1>0$ and $n \gg 1$; similarly, use
  $a_n\gtrsim b_n$ for a one sided asymptotic relation.
\end{notation}


\section{Cones and Martingales}

In order to get the right martingale representation, we begin by recalling
a few formulas concerning the transfer operator; the conditional
expectation is considered with respect to the measure $m$, and $\B$ denotes
the Borel $\sigma$-algebra on $[0,1]$. We have:
%
\[
\E[\phi \mid \TF^{-k}\B]= \frac{\P^k(\phi)}{\P^k(\one)} \circ \TF^k
\]
\[
P(\phi\circ T \cdot \psi)=\phi \cdot P(\psi)
\]
and therefore, for $0\le \ell\le k$
\[
\E[\phi\circ \TF^\ell \mid \TF^{-k}\B]=
\frac{\P_{\ell+1}^{k-\ell}(\phi \cdot \P^\ell(\one))}{\P^k(\one)}
\circ \TF^k.
\]
Recall that for $L^2(m)$-functions these conditional expectations are
the orthogonal projections in $L^2(m)$.

We denote, as in Definition~\ref{def.mean}, $\phi -m(\phi\circ\TF^j)$ by
$\mean{j}{\phi}$, with the convention that $\mean{0}{\phi}=0$. Therefore we
have for the centered sum (\ref{eq:defn-S_n}):
$S_n=\sum_{k=1}^n \mean{k}{\phi}\circ \TF^k=\sum_{k=0}^n \mean{k}{\phi}\circ \TF^k$.

Introduce
\[
\H_n\circ \TF^n :=\E(S_{n-1} \mid \TF^{-n} \B).
\]
Hence $\H_1=0$, and the explicit formula for $\H_n$ is
\begin{equation}\label{eq:h_n}
\H_n =\frac{1}{\P^n \one} \left[P_n (\mean{n-1}{\phi} \P^{n-1} \one) +P_n P_{n-1}
  (\mean{n-2}{\phi} \P^{n-2} \one)+ \dots + P_n P_{n-1} \dots P_1 (\mean{0}{\phi} \P^0
  \one)\right].
\end{equation}
It is not hard to check that setting
\[
S_n=M_n + \H_{n+1}\circ \TF^{n+1}
\]
the sequence $\{M_n\}$ is a reverse martingale for the decreasing
filtration $\{\B_n:=\TF^{-n}\B\}$:
\[
\E(M_n \mid \B_{n+1})=0.
\]
In particular,
\begin{equation}
  \label{eq:psi_n}
  M_n-M_{n-1} = \psi_n \circ \TF^n \text{\quad with \quad} \psi_n :=\mean{n}{\phi}
  +\H_n-\H_{n+1} \circ T_{n+1}.
\end{equation}

We recall three lemmas from~\cite{BC-CLT}, stated in the current context:

\begin{lemma}[{{\cite[Lemma 2.6]{BC-CLT}}}]\label{lem:c}
\[
\sigma_n^2:=\E[(\sum_{i=1}^n \mean{i}{\phi}\circ
\TF^i)^2]=\sum_{i=1}^n\E[\psi_i^2\circ \TF^i] - \int \H_1^2 +\int
\H_{n+1}^2\circ \TF^{n+1}
\]
(and $\H_1=0$).
\end{lemma}
To prove this Lemma we replace our $\H_n$ with $\omega_n$ in \cite{BC-CLT}.

\begin{lemma}[{{\cite[proof of Lemma
      3.3]{BC-CLT}}}] \label{lem:peligrad-computation}
  Let $\H_j^{\epsilon}=\H_j \one_{\{|\H_j|\le \epsilon \sigma_n\}}$, where for
  simplicity of notation we have left out the dependence on $n$. Then
\[
\int\left(\sum_{j=1}^n \psi_j \circ \TF^j \cdot \H_{j+1}^{\epsilon}\circ
  \TF^{j+1}\right)^2 =\sum_{j=1}^n\int \left(\psi_j \circ \TF^j \cdot
  \H_{j+1}^{\epsilon}\circ \TF^{j+1}\right)^2
\]
\end{lemma}
The last formula in the proof of {\cite[Lemma 2.6]{BC-CLT}} gives:
\begin{lemma}\label{lem:e}
\[
\sigma_n^2=\sum_{i=1}^n \E [ \mean{i}{\phi}^2\circ \TF^i ] + 2\sum_{i=1}^n \E
[ (\H_i\mean{i}{\phi}) \circ \TF^i ]
\]

\end{lemma}

 The following Lemma plays a crucial role all along this paper. In a
 slightly different form it was introduced and used in
 \cite[Sect. 4]{LSV}, without a proof, and subsequently in
 \cite{lom-paper}. We now give a detailed proof in a more general
 setting.
\begin{lemma} 
  \label{lem:C1*Pk}
  Assume given a $C^1$-function $\phi:[0,1]\to \R$ and $h\in
  \cone$. where the cone $\cone$ is defined with $a>1$.

  Denote by $X$ the function $X(x)=x$. If
  \begin{eqnarray*}
    \label{eq:into-the-cone-conditions}
    \lambda & \le & -|\phi'|_\infty \\
    \nu & \ge & - |\phi+\lambda X|_\infty
    \\
    \delta & \ge & \frac{a}{\alpha+1} \big(|\phi'|_\infty+ |\lambda|\big) m(h)\\
    \delta & \ge & \frac{a}{a-1} |\phi+\lambda X + \nu|_\infty m(h)
  \end{eqnarray*}
  then
  \begin{equation}\nonumber
    (\phi +\lambda X + \nu) h  + \delta \in \cone.
  \end{equation}
\end{lemma}

\begin{rmk}
  It follows immediately that if $\phi\in C^1([0,1])$ and $h\in \cone$ then
  we can use Theorem~\ref{thm:decay} and Proposition~\ref{prop:decay_Lp} to
  obtain decay of $\P^\ell(\phi h-m(\phi h))$: consider
  $\Phi:=(\phi + \lambda X + \nu )h + \delta$,
  $\Psi:= (\lambda X + \nu )h + \delta + m(\phi h)$, with constants chosen
  according to Lemma~\ref{lem:C1*Pk} so that $\Phi , \Psi \in \cone$ (by
  definition, $m(\Phi)=m(\Psi)$), and write
  \[
  \P^\ell \big(\phi\cdot h -m(\phi\cdot h)\big)=\P^\ell(\Phi-\Psi).
  \]
\end{rmk}



\begin{cor}\label{cor:C1*Pk-applications}

\let\newphi\omega

  In particular, for a sequence $\newphi_k\in C^1([0,1])$ with
  $\|\newphi_k\|_{C^1}\le K$ and $h_k\in \cone$ with $m(h_k)\le M$ (e.g,
  $h_k:=\P^k(\one)$), one can choose constants $\lambda$, $\nu$ and
  $\delta$ so that
\begin{equation}\nonumber
  (\newphi_k +\lambda X + \nu) h_k   + \delta,
  (\lambda X + \nu) h_k   + \delta + m(\newphi_k h_k)\in \cone
  \text {\quad for all $k\ge 1$}
\end{equation}
and therefore
\begin{equation}\nonumber
  ||\P^n\big(\newphi_k h_k-m(\newphi_k h_k)\big)||_1\le C_{\a,K,M} \;
  n^{-\frac{1}{\a}+1}(\log n)^{\frac{1}{\a}}
  \text{\quad for all $n\ge 1$, $k\ge 1$},
\end{equation}
where the constant $C_{\a,K,M}$ has an explicit expression in terms of
$\a, K$ and $M.$ Decay in $L^p$ now follows from
Lemma~\ref{lem:L2bounds}: if $1\le p < 1/\alpha$ then
\begin{equation}\nonumber
  ||\P^n\big(\newphi_k h_k-m(\newphi_k h_k)\big)||_p\le C_{\alpha, K, M, p} \;
  n^{-\frac{1}{p\alpha}+1}
  \text{\quad for all $n\ge 1$, $k\ge 1$}
\end{equation}
(ignoring the $\log$-correction, see Remark~\ref{rmk:no-log-factor})
where the constant on the right hand side depends now upon $p$ too.
\end{cor}

\begin{pfof}{Lemma~\ref{lem:C1*Pk}}
  Denote $\Phi:=(\phi+\lambda X + \nu) h + \delta$. There are three
  conditions for $\Phi$ to be in $\cone$.

  \underline{$\Phi$ nonnegative and decreasing.} If
  $\lambda \le -\sup \phi'$ and $\nu\ge -\inf (\phi + \lambda X)$ then
  $\phi+\lambda X + \nu$ is decreasing and nonnegative. Therefore $\Phi$ is
  also decreasing (because $h\in\cone$) and nonnegative provided
  $\delta \ge 0$.

  \underline{$\Phi X^{1+\alpha}$ increasing.} For $0 < x < y \le 1$,
  need
  \begin{eqnarray*}
    & & \big[(\phi (x)+\lambda x + \nu) h(x) + \delta\big]x^{1+\alpha} \le
    \big[(\phi (y)+\lambda y + \nu) h(y) + \delta\big]y^{1+\alpha}  \\
    &\iff& [\phi(x) + \lambda x +\nu] \le [\phi(y) + \lambda y +\nu]\frac
    {h(y)}{h(x)}\frac{y^{\alpha+1}}{x^{\alpha+1}} +
    \delta \left[\frac{y^{\alpha+1}}{x^{\alpha+1}} -1\right] \frac{1}{h(x)}
  \end{eqnarray*}
  Since $h X^{\alpha+1}\ge 0$ is increasing, $1\le \frac
  {h(y)}{h(x)}\frac{y^{\alpha+1}}{x^{\alpha+1}}$, so it suffices to
  have
  \begin{eqnarray*}
    & & \phi(x) + \lambda x +\nu \le [\phi(y) + \lambda y +\nu] +
    \delta \left[\frac{y^{\alpha+1}}{x^{\alpha+1}} -1\right] \frac{1}{h(x)}\\
    & \iff&  \delta \ge -\big[(\phi(y) + \lambda y +\nu) -
    (\phi(x) + \lambda x +\nu)\big] \frac{h(x)}{\frac{y^{\alpha+1}}{x^{\alpha+1}} -1}.
  \end{eqnarray*}
  By the mean value theorem and using that $\alpha \le 1$,
  $y^{\alpha+1}-x^{\alpha+1}=(\alpha+1)\xi^\alpha(y-x)\ge (\alpha+1)
  x^\alpha(y-x) \ge (\alpha+1) x (y-x)$; therefore
  \begin{equation*}
    0\le \frac{h(x)}{\frac{y^{\alpha+1}}{x^{\alpha+1}} -1} =
    \frac{h(x) x^{\alpha+1}}{y^{\alpha+1} - x^{\alpha+1}}
    \le
    \frac{h(x) x^\alpha}{(\alpha+1) (y - x)}
    \le \frac{a m(h)}{(\alpha+1) (y - x)}.
  \end{equation*}
  Meanwhile,
  \begin{equation*}
    -\big[(\phi(y) + \lambda y +\nu) -
    (\phi(x) + \lambda x +\nu)\big]  \le (|\phi'|_\infty+|\lambda|)(y-x).
  \end{equation*}
  Using these in the above lower bound for $\delta$, we conclude that
  it suffices to have
  \begin{equation*}
    \delta  \ge  \frac{a}{\alpha+1} \big(|\phi'|_\infty+ |\lambda|\big) m(h)
  \end{equation*}

  \underline{$\Phi X^{\alpha}\le a m(\Phi$).} Using that $h X^\alpha
  \le a m(h)$,
  \begin{eqnarray*}
    [(\phi+\lambda X + \nu) h + \delta] X^\alpha
    \le   (\phi+\lambda X + \nu) h X^\alpha + \delta
    \le  \sup (\phi+\lambda X + \nu)a m(h) + \delta.
  \end{eqnarray*}
  On the other hand, $a m((\phi+\lambda X + \nu) h + \delta)\ge a
  \inf(\phi+\lambda X + \nu) m(h) + a \delta$, so it suffices to have
\begin{eqnarray*}
  \sup  (\phi+\lambda X + \nu)a m(h) + \delta  \le a \inf(\phi+\lambda
  X + \nu) m(h) + a \delta \\
  \iff
  \delta  \ge  \frac{a}{a-1} \big[   \sup  (\phi+\lambda X + \nu) - \inf(\phi+\lambda X + \nu) \big]m(h).
\end{eqnarray*}
\end{pfof}



Note that, since the transfer operators are monotone,
\[
\left|P_n \dots P_{k+1} [\phi \P^k \one] \eval{x} \right| \le P_n \dots
P_{k+1}[|\phi|_{\infty} \P^k \one] \eval{x} = |\phi |_{\infty}P_n \dots
P_{k+1}[ \P^k \one] \eval{x}.
\]
Since $|\phi |_{\infty}P_n \dots P_{k+1}[ \P^k \one]$ lies in the cone
$\cone$ this implies that
\[
|P_n \dots P_{k+1} [\phi \P^k \one]| \eval{x} \le a |\phi|_\infty
x^{-\alpha}.
\]
The following Lemma gives control over the $L^p$-norm of functions with
such a bound.

\begin{lemma}\label{lem:L2bounds}

  Suppose that $f\in L^1(m)$ and $|f (x)| \le K x^{-\alpha}$.
  Then, provided $p \ge 1$ and $\alpha p < 1$,
  \[
  ||f||_p\le C_{\alpha,p} ||f||_1^{\frac{1-\alpha p}{p-p\alpha}}
  K^{\frac{p-1}{p - p\alpha}}
  \]

  In particular, if $|f (x)| \le K x^{-\alpha}$ and
  $||f||_1 \le M n^{1-\frac{1}{\alpha}}$, then
  \[
  ||f||_p\le C_{K, M, \alpha,p} n^{1-\frac{1}{p\alpha}} \text{\ for
    $1\le p < 1/\alpha$.}
  \]
  Therefore, for $1\le p < 1/(2\alpha)$, there is $\delta >0$ such that
  $||f||_{p}\le C_{K, M,\alpha, p} n^{-1-\delta}$.
\end{lemma}

\begin{proof}
  %
  The case $p=1$ is obviously true, so we assume from now on that
  $p>1$.  Denote $C_1:=||f||_1$.  Compute, for $0<x_* \le 1$, and
  $\alpha p < 1$: $\int_{x_*}^1 |f|^{p}dx \le \sup \{ |f(x)|^{p-1}
  \mid x_*\le x \le 1 \} \int_0^1 |f| dx \le K^{p-1}
  x_*^{-\alpha(p-1)} C_1$, and $\int_{0}^{x_*} |f|^{p}dx\le
  K^{p}\int_{0}^{x_*} x^{-\alpha p}dx = \frac{K^{p}}{1-\alpha p}
  x_*^{1-\alpha p}$.  We want to minimize over $x_*$ the quantity
  \[
  G(x_*):=K^{p-1} C_1 x_*^{-\alpha(p-1)} + K^{p} \frac{1}{1-\alpha p}
  x_*^{1-\alpha p} = A x_*^{-\alpha(p-1)} + B x_*^{1-\alpha p}.
  \]
  It reaches its minimum value for $x_*^{\alpha-1}=\frac{B(1-\alpha
    p)}{A \alpha (p-1)}$, which gives for the minimum of $G^{1/p}$ the
  value
  \[
  C_{\alpha, p} C_1^{\frac{1-\alpha p}{1-\alpha}\frac {1} {p}}
  K^{\frac{p-1}{p}\frac{1}{1-\alpha}}.
  \]

  For the last statement notice that
  $\frac{1-p\alpha}{p\alpha}>1 \iff 0<\alpha p <1/2$.
\end{proof}\\


\begin{cor}\label{cor:h-bounded}
  We have:
  \begin{enumerate}
  \item $||\H_n||_{q}$ is uniformly bounded in $n$ for $1\le
    q<\frac{1}{2\a}.$
  \item $||\H_n\circ \mathcal{T}^n||_{r}$ is uniformly bounded in $n$
    for $1\le r<\frac{1}{2\a}-\frac{1}{2}$.
\end{enumerate}
\end{cor}

\begin{proof}
  Recall that $\H_n$ is given in \eqref{eq:h_n}. By~\cite[Remark
  1.3]{lom-paper}, $\P^n(\one)\ge D_\alpha>0$ on $(0,1]$. We now apply
  Minkowski's inequality in the sum defining $\H_n$. Thanks to
  Lemma~\ref{lem:L2bounds} each term of the form $P_n P_{n-1}\dots
  P_{n-\ell} (\mean{n-\ell-1}{\phi} \P^{n-\ell-1} {\bf 1}), \ \ell\in[0,n-1]$ will
  be bounded in $L^p$ by $\frac{2}{D_{\a}}\ C_{\a, K,p} \ \ell
  ^{1-\frac{1}{p\a}},$ where $K$ is the $C^1$ norm of $\phi$. The role
  of $h_k$ in Lemma \ref{cor:C1*Pk-applications} is now played by
  $\P^{n-\ell-1} \one$ and therefore $M=1.$ By summing over $\ell$
  from $1$ to infinity, we get a convergent series whenever $p\a<1/2.$
  We now write $\int |\H_n\circ\TF^n|^rdx=\int |\H_n|^r \P^n{\one} \
  dx.$ Since $\P^n{\one}$ belongs to $L^p(m)$ for $1\le
  p<\frac{1}{\a}$ by the definition of $\cone$ and its invariance
  property,
  it suffices that the function $|\H_n|^{r\frac{p}{p-1}}$ be uniformly in
  $L^1(m)$, and therefore, by the previous item, that
  $r\frac{p}{p-1}<\frac{1}{2\a}.$ Thus it suffices to have
  $1\le r<\frac{p-1}{2p\a}$ for some $1\le p<\frac{1}{\a}$, which means
  $1\le r<\frac{1}{2\a}-\frac{1}{2}$.
\end{proof}


As we said in the Introduction, we will also have a pointwise bound on
the $\H_n$'s.

\begin{lemma}\label{lem:pointwise-bound-for-H_n}
  For $0 < \a < 1/2$, there is a  constant $C$ depending on $\a$ and $K=||\phi||_{C^1}$,  such that
  \begin{equation}\label{eq:pointwise-bound}
    |\H_n(x)|\le C x^{-\a-1} \quad \text{for all $x\in (0,1]$, $n\ge
      1$}.
  \end{equation}
\end{lemma}

\begin{proof}
  By using again formula~\eqref{eq:h_n} for $\H_n$ (where $\phi_0=0$)
  and the bound $\P^n(\one)\ge D_\alpha>0$ we are left with the
  pointwise estimate of
  \begin{equation}\nonumber 
    P_n (\mean{n-1}{\phi} \P^{n-1} \one) +P_n P_{n-1}
    (\mean{n-2}{\phi} \P^{n-2} \one)+ \dots + P_n P_{n-1} \dots P_1 (\mean{0}{\phi} \P^0
    \one).
  \end{equation}
  By Corollary~\ref{cor:C1*Pk-applications}, for each $k\ge 1$ one can
  write $\mean{k}{\phi}\P^k\one = \big(\phi
  -m(\phi\circ\TF^k)\big)\P^k\one=A_k-B_k$ where $A_k, B_k\in \cone$
  with $m(A_k), m(B_k)$ uniformly bounded by some constant $ C_{\a, K}
  <\infty$. Therefore, by the decay Theorem~\ref{thm:decay} (and
  ignoring the $\log$-correction), there is a new constant $C'$
  depending only on $\a$ and $K$ such that
  \begin{equation}\label{eq:Ak-Bk_L1-bound}
    \|\P_{k+1}^{n-k}(A_k-B_k)\|_1 \le C'
    (n-k)^{-\frac{1}{\alpha}+1}.
  \end{equation}

  We now recall the footnote to the proof of \cite[Lemma 2.3]{LSV}: if
  $f\in \cone$ with $m(f)\le M$ then
    \begin{equation}\label{eq:lip-norm}
      |x^{\a+1} f(x)-y^{\a+1} f(y)| \le a(1+\a)M|x-y| \quad \text{for
        $0<x,y\le 1$}.
    \end{equation}
    But a bound $|g(x)-g(y)|\le L|x-y|$ for the Lipschitz-seminorm
    $|g|_{\Lip}$ implies
  \begin{equation}\label{eq:lip-integral}
    \|g\|_1\ge C_L \|g\|_\infty.
  \end{equation}

  Combining the above observations and
  since $m(\P_{k+1}^{n-k}(f))=m(f)$, we obtain that
  $|X^{\a+1}\P_{k+1}^{n-k}(A_k-B_k)|_{\Lip} \le
  |X^{\a+1}\P_{k+1}^{n-k} (A_k)|_{\Lip}+|X^{\a+1}\P_{k+1}^{n-k}
  (B_k)|_{\Lip}\le L$ uniformly for $n\ge 1$, $1\le k < n$, and then
  \[
  \|X^{\a+1}\P_{k+1}^{n-k}(A_k-B_k)\|_\infty \le 1/C_L \|X^{\a+1}
  \P_{k+1}^{n-k}(A_k-B_k)\|_1 \le C'' (n-k)^{-\frac{1}{\alpha}+1}
  \]
  for a new constant $C''$ depending only on $\a, K, L,$ which implies
  that
  \[
  |\P_{k+1}^{n-k}(A_k-B_k)(x)| \le x^{-\a-1} C'' (n-k)^{-\frac{1}{\alpha}+1}
  \]
  and therefore, for $0<\alpha < 1/2$,
  \[
  \left|\sum_{k=1}^{n-1} \P_{k+1}^{n-k}(A_k-B_k)(x)\right| \le
  x^{-\a-1} C'' \sum_{k=1}^{n-1} (n-k)^{-\frac{1}{\alpha}+1} \le C
  x^{-\a-1}
  \]
  as desired.
\end{proof}

We finish this Section by proving a type of Borel-Cantelli Lemma which
is an unavoidable tool in proving non-stationary limit theorems.


\begin{thm}[Strong Borel-Cantelli]\label{thm:SBC}
  Suppose that for $j\ge 1$,
  $\psi_j \in C^1([0,1])$ with uniformly bounded $C^1$-norms.

  (a) If $0 < \alpha < 1/2$ then
  \[
  \sum_{j=1}^n \psi_j (\TF^j) -\sum_{j=1}^n m(\psi_j (\TF^j))= O
  (n^{1/2}(\log\log n)^{3/2}) \text{\quad $m$-a.e.}
  \]
  and therefore, if $\liminf_j m(\psi_j\circ \TF^j)>0$ then
  \[
  \frac{\sum_{j=1}^n \psi_j (\TF^j x)}{\sum_{j=1}^n m(\psi_j \circ
    \TF^j)}\to 1 \text{\quad $m$-a.e. $x$.}
  \]

  (b) If $0<\a<1$ then
  \[
  \frac{1}{n} {\left[\sum_{j=1}^n \psi_j (\TF^j x) - \sum_{j=1}^n
      m(\psi_j \circ \TF^j)\right]}\to 0 \text{\quad $m$-a.e. $x$.}
  \]
\end{thm}

\begin{proof}
  To prove the first statement in part (a) we will use the {\Sprindzuk}
  Theorem~\ref{thm:sprindzuk} in the Appendix. By adding the same constant
  to all the $\psi_j$'s and rescaling, we can assume without loss of
  generality that $\inf_j m(\psi_j\circ \TF^j) > 0$ and
  $\sup_j m(\psi_j\circ \TF^j)\le 1$. We take $g_k=m(\psi_k\circ \TF^k)$
  and $h_k=1$ in Theorem~\ref{thm:sprindzuk}, thus it suffices to give a
  linear upper bound for $\E[(\sum_{j=1}^n \psi_j\circ \TF^j -b_n )^2]$,
  where $b_n:=\sum_{j=1}^n m(\psi_j\circ \TF^j)$; note that the same
  estimate can be derived for sums over $m \le j \le n$. Expand
  \begin{multline*}
    \E[(\sum_{j=1}^n \psi_j \circ \TF^j -b_n )^2] =\sum_{j=1}^n\E [\psi_j
    \circ \TF^j -m(\psi_j \circ \TF^j)]^2\\
    +2\sum_{i=1}^n \sum_{j>i} \E[(\psi_j \circ \TF^j-m(\psi_j \circ
    \TF^j)(\psi_i \circ \TF^i-m(\psi_i \circ \TF^i))]
  \end{multline*}
  and use the decay to estimate the mixed terms. Denote, following
  Definition~\ref{def.mean}, $\mean{j}{g}:=g-m(g\circ \TF^j)$. Then, for
  $j>i\ge 1$,
    \begin{multline*}
      |\E[(\psi_j\circ \TF^j-m(\psi_j\circ \TF^j)(\psi_j\circ \TF^j-m(\psi_j\circ \TF^j)]| = |\E[\mean{j}{\psi_j} \circ \TF^j \cdot
      \mean{i}{\psi_i} \circ
      \TF^i]| \\
      = |\E[(\mean{j}{\psi_j} \circ \TF_{i+1}^{j-i} \cdot
      \mean{i}{\psi_i} \cdot \P^i(\one)]| = |\E[(\mean{j}{\psi_j} \cdot
      \P_{i+1}^{j-i} (\mean{i}{\psi_i}
      \P^i(\one))]| \\
      \le \|\mean{j}{\psi_j}\|_\infty \|\P_{i+1}^{j-i} (\mean{i}{\psi_i}
      \P^i(\one))\|_1 \le C (j-i)^{1-\frac 1 \alpha}
    \end{multline*}
    where in the last inequality we used
    Corollary~\ref{cor:C1*Pk-applications}.
    Therefore
    \begin{multline*}
      \E[(\sum_{j=1}^n \psi_j\circ \TF^j -b_n )^2] \\ \le 2\sum_{i=1}^n
      |(\psi_j\circ \TF^j-m(\psi_j\circ \TF^j)|_{\infty} m(\psi_j\circ \TF^j) + 2 C
      \sum_{i=1}^n \sum_{j>i} (j-i)^{1-\frac{1}{\alpha}} \le n C',
    \end{multline*}
    where the constants $C, C'$ are independent of $j$ and $n.$. The
    conclusion now follows from the {\Sprindzuk}
    Theorem~\ref{thm:sprindzuk}.

    For (b), note that for $1/2 \le \a < 1$ the above computation still
    gives
    \begin{equation*}
      \E[(\sum_{j=1}^n \psi_j\circ \TF^j -b_n )^2]\le C
      n^{3-\frac{1}{\alpha}}
    \end{equation*}
    which implies that
    \[
    \sum_{j=1}^n \psi_j\circ \TF^j -b_n = O(n^{1-\eta}) \text{ a.s.}
    \]
    for some $\eta>0$, see the standard Lemma~\ref{lem:SLLN}.
  \end{proof}

  \begin{lemma}\label{lem:SLLN}
    Assume the random variables $X_n$ have mean zero, and there are
    $M<\infty$, $\gamma <2$ such that
    \[
    \|X_n\|_\infty\le M, \quad \Var\big(\sum_{k=1}^n X_k\big) \le C
    n^\gamma \qquad \text { for all $n$.}
    \]
    Then
    \begin{equation*}
      \sum_{k=1}^n X_k = O(n^{\eta}) \text { a.s. for  $\eta > \frac{\gamma+1}{3}$.}
    \end{equation*}
  \end{lemma}

  \begin{proof}
    Denote $S_n:=\sum_{k=1}^n X_k$. From Tchebycheff's inequality,
    \begin{equation*}
      P(|S_n| > n^{1-\delta}) \le \frac{\Var(S_n)}{(n^{1-\delta})^2} \le C
      n^{\gamma-2 \delta -2}.
    \end{equation*}
    Pick $\delta>0$ so that $\gamma-2 \delta -2<0$ and $\omega>0$ such that
    $\omega (2-\gamma+2 \delta)>1$. Then, for the subsequence
    $n_k:=k^\omega$,
    \begin{equation*}
      \sum_k P(|S_{n_k}| > n_k^{1-\delta}) < \infty
    \end{equation*}
    so, by Borel-Cantelli,
    \begin{equation}\label{eq:subsequence}
      |S_{n_k}| = O(n_k^{1-\delta}) \text { a.s.}
    \end{equation}
    Using \eqref{eq:subsequence}, one has a.s.: if $n_k \le n < n_{k+1}$
    for some $k$, then
    \[
    |S_n|\le |S_{n_k}| + [n_{k+1}-n_k] \sup \|X_\ell\|_\infty \le
    O(n_k^{1-\delta})+ C k^{\omega-1} M \le O(n^{1-\delta})+
    C(n^{1/\omega})^{\omega-1} M
    \]
    therefore $|S_n|=O(n^\eta)$ a.s. with
    \[
    \eta=\max\left\{1-\delta,\frac{\omega-1}{\omega}\right\}.
    \]
    Optimize over $\delta$ and $\omega$ to get the claimed lower bound on
    $\eta$.
  \end{proof}

\section{Central Limit Theorem}

We assume in this section that $0 < \a < 1/2$ (note that in the
stationary case the CLT holds only in this range).  With our approach
we can only prove the non-stationary CLT for a lower upper bound on
$\a$, which will be stated later.

We define scaling constants $\sigma_n^2=\E[(\sum_{j=1}^{n}\mean{j}{\phi} \circ
\mathcal{T}^j)^2]$. This sequence of constants play the role of non-stationary
variance. As we pointed out in the Introduction, giving estimates on the growth and non-degeneracy of $\sigma_n$
in this non-stationary setting is more difficult than in the usual
stationary case.

\begin{thm}[CLT for $C^1$ functions]
  \label{thm:CLT}
  Let $\phi$ be a $C^1([0,1])$ function, and define $S_n$ as
  in~\eqref{eq:defn-S_n},
  \begin{equation}\nonumber
    S_n:=\sum_{k=1}^{n}\mean{k}{\phi} \circ T_{\b_k}\circ \cdots \circ T_{\b_1}.
  \end{equation}

  Assume that
  \[
    \sigma_n^2:=\Var(S_n)=\E[(\sum_{i=1}^n \mean{i}{\phi}\circ \TF^i)^2]
    \gtrsim  n^{\beta}.
  \]

  Then
  %
  \[
    \text{$0 <\a < \alphaupperbound$ and $\b > \frac{1}{2(1-2\a)}$}
    \quad\implies\quad \frac{S_n}{\sigma_n} \to^{d} \mathcal{N}(0,1).
  \]

  In particular, $\b>9/14=0.643$ suffices for any
  $0< \a < \alphaupperbound$, and the lower bound on $\b$ approaches
  $\frac 1 2$ as $\a$ approaches zero.
\end{thm}


\begin{rmk}
   The above Theorem holds, with the same proof, if we allow $\phi$ to vary
   but stay bounded in $C^1$ (as in our Strong Borel-Cantelli
   Theorem~\ref{thm:SBC}). That is, consider
   \[
   S_n:=\sum_{k=1}^{n}\mean{k}{\phi_k} \circ T_{\b_k}\circ \cdots \circ
   T_{\b_1}
   \]
   where $\phi_k\in C^1([0,1])$ have uniformly bounded $C^1$-norms.

   To keep the notation simpler, we will not prove this more general case.
\end{rmk}

Following the approach of Gordin we will express
$S_n=\sum_{j=1}^n \mean{j}{\phi} \circ \TF^j$ as the sum of a
(non-stationary) martingale difference array and a controllable error term
and then use the following Theorem from Conze and Raugi~\cite[Theorem
5.8]{Conze_Raugi}, which is a modification of a result of
B.~M.~Brown~\cite{Brown} from martingale differences to reverse martingale
differences.

\begin{thm}[{{\cite[Theorem 5.8]{Conze_Raugi}}}]\label{thm:Brown}
  Let $(X_i, \mathcal{B}_i)$ be a sequence of differences of square
  integrable reversed martingales, defined on a probability space
  $(\Omega, \mathcal{B}, \mathcal{P})$. For $n\ge 0$ let
\[
  S_n=X_0+\ldots + X_{n-1},~\sigma_n^2=\sum_{k=0}^{n-1}
  \E[X_k^2],~V_n=\sum_{k=0}^{n-1} \E[X_k^2| \mathcal{B}_{k+1}].
\]
Assume the following two conditions hold:
\begin{itemize}
\item[(i)] the sequence of random variables $(\sigma_n^{-2} V_n)_{n\ge1} $ converges in probability to $1$.
\item[(ii)] For each $\epsilon>0$,
  $\lim_{n\to \infty} \sigma_n^{-2} \sum_{k=0}^{n-1} \E[X_k^2 \one_{\{
    |X_k|> \epsilon \sigma_n\}} ]=0.$
\end{itemize}
Then
\[
\lim_{n\to \infty} \sup_{\alpha \in \R} \left|P
  \left[\frac{S_n}{\sigma_n}<a\right]-\frac{1}{\sqrt{2\pi}}
  \int_{-\infty}^{\alpha} e^{-\frac{x^2}{2}}~dx\right|=0.
\]
\end{thm}









\begin{pfof}{Theorem~\ref{thm:CLT}}

We will apply Theorem~\ref{thm:Brown} with the following identifications:
\begin{itemize}
\item $X_n=\psi_n \circ \TF^n.$
\item $\mathcal{B}_{n}=\TF^{-n}\mathcal{B}.$
\item $\sigma_n^2=\E[(\sum_{i=1}^n \mean{i}{\phi}\circ \TF^i)^2]$ as
  defined earlier, but if $\alpha<\frac15$ then $\sigma_n^2=
  \E[(\sum_{i=1}^n \psi_i\circ \TF^i)^2]+O(1)$ by Lemma~\ref{lem:c} and
  Corollary~\ref{cor:h-bounded}.
\end{itemize}




Let us take $\H_n$ defined in~\eqref{eq:h_n} and $\psi_n$ given
in~\eqref{eq:psi_n}
\[
\psi_n:=\mean{n}{\phi} +\H_n-\H_{n+1} \circ T_{n+1}.
\]




Recall that $\psi_n \circ \TF^n$ is a reverse martingale difference scheme,
{\optbf uniformly bounded in $L^{r_1}(m)$ provided
  $1\le r_1<\frac{1}{2\a}-\frac{1}{2}$ (because so is
  $\H_{k}\circ \TF^{k}$, see the second item in Corollary
  \ref{cor:h-bounded}).}
Once we establish $(i)$ and $(ii)$ it follows that $\lim_{n\to \infty}
\frac{1}{\sigma_n} \sum_{j=1}^{n} \psi_j \circ \TF^j \to \mathcal{N}(0,1)$
in distribution. Finally, since $[\sum_{j=1}^{n} \mean{j}{\phi} \circ
  \TF^j]-[\sum_{j=1}^{n} \psi_j \circ \TF^j]=\H_{n+1}\circ \TF^{n+1}$ is
uniformly bounded in $L^{2}$ if $\a<1/5$, we conclude that
%
$\lim_{n\to \infty} \frac{1}{\sigma_n} \sum_{j=1}^{n} \mean{j}{\phi} \circ
\TF^j \to \mathcal{N}(0,1)$ in distribution as well.

We will now verify conditions $(i)$ and $(ii)$ of Theorem~\ref{thm:Brown}.
We defer to the end of this proof the discussion about the possible choices
for~$\a$ and $\b$, see \eqref{eq:alpha-beta-bounds}.

For condition $(ii)$ we begin by noticing that the functions $(\psi_n\circ
\TF^n)^2$ have a uniformly bounded $L^p$-norm if the same is true
for $(\H_{n+1}\circ \mathcal{T}_{n+1})^2$; this holds provided
$1\le 2p<\frac{1}{2\a}-\frac{1}{2}$, and we also need $p>1$ (for a H\"older
inequality, see below). By Minkowski's inequality,
$\|(\psi_n\circ \TF^n)^2\|_{L^p(m)}$ will therefore be bounded uniformly in
$n$ by some constant $\hat{C}_p$. Then we have by H\"older's and
Tchebycheff's inequality, where $1/p+1/q=1$:
\begin{multline}
  \sigma_n^{-2} \sum_{k=0}^{n-1} \E[(\psi_k\circ \TF^k)^2 \one_{\{
      |\psi_k\circ \TF^k)|> \epsilon \sigma_n\}} ]\le \sigma_n^{-2}
  \ \sum_{k=0}^{n-1} \|(\psi_k\circ \TF^k)^2\|_p m(|\psi_k\circ \TF^k|>
  \epsilon \sigma_n)^{\frac1q} \\ \le \sigma_n^{-2}
 \sum_{k=0}^{n-1} \|(\psi_k\circ \TF^k)^2\|_p \left[\frac{1}{(\epsilon
     \sigma_n)^s} \E( |\psi_k\circ \TF^k|^s)\right]^{\frac{1}{q}}
\\
\label{eq.clt-ii-upper-bound}
\le \sup_k \|\psi_k\circ \TF^k\|_{2p}^2 \sup_k \|\psi_k\circ
\TF^k\|_{s}^{\frac s q}
\frac{n}{\epsilon^{\frac{s}{q}}\sigma_n^{2+\frac{s}{q}}} \le C
\frac{n}{\epsilon^{\frac{s}{q}}\sigma_n^{2+\frac{s}{q}}}
\end{multline}
if $1\le s< \frac{1}{2\alpha}-\frac12:=\tilde{s}(\alpha).$
Since $q=(1-1/p)^{-1}>\frac{1-\alpha}{1-5\alpha}:=\tilde{q}(\alpha)$
provided $\alpha<\frac15$, the largest value we can use for the exponent of
$\sigma_n$ is
$2+\frac{s}{q}=2+\frac{\tilde{s}(\alpha)-\iota}{\tilde{q}(\alpha)+\iota},$
for $0<\iota$ small. {\optbf If we now assume that the variance grows as
  $\sigma_n^2\gtrsim  n^{\beta}$, then we need
  $\beta>\frac{2(\tilde{q}(\alpha)+\iota)}{2\tilde{q}(\alpha)+\tilde{s}(\alpha)+\iota}:=b_{\alpha}(\iota),$
  in order for the upper bound~\eqref{eq.clt-ii-upper-bound} to vanish as
  $n$ tends to infinity.} It is easy to check that when $\tilde{q}(\alpha)$
and $\tilde{s}(\alpha)$ are positive then the function $\iota\mapsto
b_{\alpha}(\iota)$ is decreasing for $\iota>0$, so suffices to require that
$\beta> b_{\alpha}(0)= \frac{4\alpha}{1-\alpha}.$

The hard part lies in establishing~$(i)$. This is in contrast with the
stationary setting where condition $(i)$ is usually a straightforward
consequence of the ergodic theorem.

For $(i)$, we first prove that
\begin{equation}\label{eq:condition_i_main-part}
  \frac1{\sigma_n^2}\sum_{j=1}^n \psi_j^2\circ \TF^j\rightarrow 1
  \text{\quad in probability as $n\rightarrow\infty$.}
\end{equation}
That~\eqref{eq:condition_i_main-part} implies $(i)$ follows from
Theorem~\ref{bU}.

We follow~\cite[Lemma 3.3 and proof of Theorem 3.1 (II)]{BC-CLT}, which
uses an argument of Peligrad~\cite{Peligrad}.
Since \mbox{$\psi_j=\mean{j}{\phi}+\H_j-\H_{j+1}\circ T_{j+1}$},
\begin{eqnarray*}
  \psi_j^2 &=&\mean{j}{\phi}^2+2\mean{j}{\phi} \H_j +\H_j^2+\H_{j+1}^2\circ
               T_{j+1}-2\H_{j+1}\circ T_{j+1}(\mean{j}{\phi} +\H_j)
  \\
           &=& \mean{j}{\phi}^2 +2\mean{j}{\phi} \H_j +\H_j^2+\H_{j+1}^2\circ
               T_{j+1}-2\H_{j+1}\circ T_{j+1} (\psi_j +\H_{j+1}\circ T_{j+1})
  \\
           &=& \mean{j}{\phi}^2  +(\H_j^2 -\H_{j+1}^2 \circ T_{j+1}) -2\psi_j \cdot
               \H_{j+1} \circ T_{j+1}  +2\mean{j}{\phi}  \H_j.
\end{eqnarray*}

Therefore
\begin{eqnarray*} 
  \sum_{j=1}^n \psi_j^2\circ \TF^j &=&
 \left(\H_1^2\circ \TF_{1} - \H_{n+1}^2 \circ \TF_{n+1}\right)
                                     -
                                     \left[\sum_{j=1}^n \psi_j\circ \TF^j
                                     \cdot \H_{j+1} \circ \TF^{j+1}\right]
  \\
                                 &+&
                                     \left[\sum_{j=1}^n \mean{j}{\phi}^2\circ \TF^j\right]
                                     +
                                     2  \left[\sum_{j=1}^n (\mean{j}{\phi} \cdot \H_{j}) \circ \TF^{j}\right].
\end{eqnarray*}

By Corollary~\ref{cor:h-bounded}, $\H_n\circ \mathcal{T}^n$ is uniformly
bounded in $L^{2}$ for $\a < \frac 1 5$ , so $\frac{1}{\sigma_n^2}
\H_{n+1}^2\circ \TF^{n+1} \to 0$ in probability.

Next we show that
\begin{equation} \label{eq:psi_j.h_j+1}
  \frac1{\sigma_n^2}\left[\sum_{j=1}^n \psi_j\circ \TF^j
    \cdot \H_{j+1} \circ \TF^{j+1}\right]
  \rightarrow 0 \text{\qquad in probability.}
\end{equation}

Define
\[
\H_j^{\epsilon}:=\H_j \one_{\{ |\H_j|\le \epsilon \sigma_n\}}.
\]
By Lemma~\ref{lem:peligrad-computation},
\[
U_n^2:=\int \left(\sum_{j=1}^n[ \psi_j\circ\TF^j \cdot
  \H_{j+1}^{\epsilon}\circ \TF^{j+1}]\right)^2 = \int
\sum_{j=1}^n[\psi_j\circ\TF^j \cdot \H_{j+1}^{\epsilon}\circ \TF^{j+1}]^2.
\]
Hence, using Lemma~\ref{lem:c} for the equality in the next computation
(note that $\H_{k}\circ \TF^{k} \in L^2$ if $\a < \frac 1 5$),
\begin{multline}\label{eq:V_n-bound}
  U_n^2\le
  \epsilon^2 \sigma_n^2 \sum_{j=1}^n \int \psi_j^2\circ \TF^j \\
  = \epsilon^2 \sigma_n^2 \left[\int (\sum_{j=1}^n \mean{j}{\phi}\circ\TF^j)^2
    + \int \H_1^2\circ \TF^1 - \int \H_{n+1}^2\circ \TF^{n+1}\right]
  \le \epsilon^2 \sigma_n^4.
\end{multline}
%
For any $a>\epsilon$ we obtain, using Tchebycheff's inequality in the third
and fourth lines below, the inequality \eqref{eq:V_n-bound}, and that
$\H_j\circ \TF^j$ is uniformly bounded in $L^{r}$ by some constant
$\hat{D}$ (Corollary~\ref{cor:h-bounded})
\begin{eqnarray*}
  m\left(\left|\frac1{{\sigma_n^2}}\sum_{j=1}^n \psi_j\circ \TF^j \cdot \H_{j+1}
      \circ \TF^{j+1}\right|>a\right)\hspace{-7cm}\\
  &\le &m\left(\max_{1\le j\le n} \left|\H_{j+1}\circ
      \TF^{j+1}\right|>\epsilon \sigma_n\right)+
  m\left(\left|\frac1{\sigma_n^2}\sum_{j=1}^n \psi_j \circ \TF^j
      \cdot \H_{j+1}^{\epsilon}\circ \TF^{j+1}\right|>a\right)\\
  &\le&
  \sum_{j=1}^n m(|\H_{j+1}\circ \TF^{j+1}|>\epsilon \sigma_n)
  +\frac{1}{a^2 \sigma_n^4}  U_n^2\\
  &\le&
  \frac{n}{(\epsilon \sigma_n)^{r}} \left({\max_{1\le j\le n} \int
    |\H_{j+1}\circ \TF^{j+1}|^{r}}\right) +\frac{\epsilon^2}{a^2} \le
  \frac{n \hat{D}^r}{(\epsilon \sigma_n)^{r}} +\frac{\epsilon^2}{a^2}.
\end{eqnarray*}
Take $a=\sqrt{\epsilon}$; {\optbf if we use that $\sigma_n^2\gtrsim
  n^{\beta}$, then $\beta>\frac{2}{r}$ with $1\le
  r<\frac{1}{2\a}-\frac{1}{2}$}, that is $\beta > \frac{4\a}{1-\a}$, allows
us to obtain~\eqref{eq:psi_j.h_j+1}.

Finally, we show that
\begin{equation}\label{eq:SB2}
  \frac1{\sigma_n^2}\sum_{j=1}^n (\mean{j}{\phi}^2 +
  2\mean{j}{\phi} \H_j)\circ \TF^j\rightarrow 1 \text{\quad in probability}.
\end{equation}

%
%
%
%
%
%

We know from our Strong Borel-Cantelli Theorem~\ref{thm:SBC} that
%
%
\begin{equation}\label{eq:SB2a}
  \sum_{j=1}^n \mean{j}{\phi}^2\circ \TF^j = \sum_{j=1}^n \E[\mean{j}{\phi}^2\circ \TF^j] +
  o(n^{\frac 1 2 + \epsilon}) \text{\quad $m$-a.e.}
\end{equation}


We will show in Lemma~\ref{lem:phi-H} that
\begin{equation}\label{eq:SB3}
\frac{1}{\sigma_n^2}\left(\sum_{j=1}^n (\mean{j}{\phi} \H_j)\circ \TF^j
  -\sum_{j=1}^n \E[ (\mean{j}{\phi} \H_j)\circ \TF^j]\right) \to 0 \text{\ in
  probability}.
\end{equation}

In view of Lemma \ref{lem:e}, equations \eqref{eq:SB2} and
\eqref{eq:SB3} imply
$\frac{1}{\sigma_n^2} [\sum_{j=1}^n \mean{j}{\phi}^2\circ
\TF^j+2\sum_{j=1}^n (\mean{j}{\phi} \H_j)\circ \TF^j ] \to 1$ in
probability.
\end{pfof}

\begin{lemma}\label{lem:phi-H}
  For $0 < \alpha < 1/5$ and the variance growing as
  $\sigma_n^2\gtrsim  n^{\beta}$ with $\beta>\frac{1}{2(1-2\alpha)}$, we
  have
  \begin{equation}\nonumber
    \frac{1}{\sigma_n^2}\left(\sum_{j=1}^n (\mean{j}{\phi} \H_j)\circ \TF^j
      -\sum_{j=1}^n \E[ (\mean{j}{\phi} \H_j)\circ \TF^j]\right) \to 0
    \text{\ in probability}.
  \end{equation}
\end{lemma}

\begin{proof}
  Write $S_n= \sum_{j=1}^n (\mean{j}{\phi} \H_j)\circ \TF^j$ and
  $E_n= \sum_{j=1}^n \E[ (\mean{j}{\phi} \H_j)\circ \TF^j]$ and estimate
\begin{eqnarray*}
  \E( |S_n-E_n|> \sigma_n^2 \epsilon)&=&\E( |S_n-E_n|^2> \sigma_n^4
                                        \epsilon^2)
  \\
     &\le& \frac{1}{\sigma_n^4 \epsilon^2}\E(|S_n-E_n|^2).
\end{eqnarray*}

When we expand $\E(|S_n-E_n|^2)$ we have, as usual, the diagonal terms and
a double summation of off-diagonal terms:
\begin{multline}\nonumber
  \E(|S_n-E_n|^2)=\sum_{j=1}^n \E( (\mean{j}{\phi} \H_j)\circ \TF^j -m [(\mean{j}{\phi}
  \H_j)\circ \TF^j)]^2)
  \\
  + 2\sum_{j=1}^n \sum_{i=1}^{j-1}  \int [(\mean{j}{\phi} \H_j)\circ \TF^j -m ((\mean{j}{\phi}
  \H_j)\circ \TF^j)][(\mean{i}{\phi} \H_i)\circ \TF^i -m ((\mean{i}{\phi} \H_i)\circ
  \TF^i)]dx.
\end{multline}
The sum of diagonal terms is $O(n)$ as
$(\mean{j}{\phi}\H_j )\circ \TF^j \in L^2 (m)$ with uniformly bounded norm
if $\a < 1/5$. {\optbf Therefore, if $\sigma_n^2 \approx n^{\beta},$ then
  the exponent $\beta$ must verify $\beta>1/2.$}

We now consider
\begin{eqnarray*}
  \sum_{j=1}^n \sum_{i=1}^{j-1}  \int [(\mean{j}{\phi} \H_j)\circ \TF^j -m ((\mean{j}{\phi}
  \H_j)\circ \TF^j)][(\mean{i}{\phi} \H_i)\circ \TF^i -m ((\mean{i}{\phi} \H_i)\circ
  \TF^i)]dx
  \\
  = \sum_{j=1}^n \sum_{i=1}^{j-1}  \int [\mean{j}{\phi} \H_j -m ((\mean{j}{\phi}
  \H_j)\circ \TF^j)]\circ \TF^j \cdot [\mean{i}{\phi} \H_i- m ((\mean{i}{\phi} \H_i)\circ
  \TF^i)]\circ
  \TF^i dx
  \\
  = \sum_{j=1}^n \sum_{i=1}^{j-1}  \int [\mean{j}{\phi} \H_j -m
  ((\mean{j}{\phi} \H_j)\circ \TF^j)]\circ \TF_{i+1}^{j-i} \cdot [\mean{i}{\phi} \H_i -m ((\mean{i}{\phi}
  \H_i)\circ \TF^i )] \cdot \P^i
  \one~dx
  \\
  =\sum_{j=1}^n \sum_{i=1}^{j-1}  \int [\mean{j}{\phi} \H_j-m ((\mean{j}{\phi} \H_j)\circ
  \TF^j)] \cdot  \P_{i+1}^{j-i} [ \H_i \mean{i}{\phi} \P^i \one -  m
  ((\mean{i}{\phi} \H_i)\circ \TF^i)\P^i \one ]~dx.
\end{eqnarray*}

%

We will prove in Lemma \ref{lem:decay2} below that $\a < 1/2$ implies
$||\P_{i+1}^{j-i}[\P^i \one \H_i \mean{i}{\phi}- \P^i \one m
  ((\mean{i}{\phi} \H_i)\TF^i)]||_2 \le \frac{C^* i}{(j-i)^{\a^*}},$ where
$C^*$ is a constant depending only on $\a$ and the $C^1$ norm of $\phi$
(and uniform in $i$ and $j$). {\optbf Here the numerator $i$ comes about as
  $1\le i\le j-1$ and $\alpha^*=\frac{1-2\a}{2\a}$ follows from the decay
  Theorem~\ref{thm:decay} and Lemma~\ref{lem:L2bounds}.
  Note also that $||(\mean{j}{\phi} \H_j)-m ((\mean{j}{\phi} \H_j)\circ
  \TF^j)||_2$ is uniformly bounded in $j$ provided $\a<\frac{1}{4}$}, see
Corollary~\ref{cor:h-bounded}.

We have to show that each row summation satisfies
\[
|\sum_{i=1}^{j-1} \int [(\mean{j}{\phi} \H_j)-m ((\mean{j}{\phi} \H_j)\circ \TF^j)]
\P_{i+1}^{j-i} [\P^i \one \H_i \mean{i}{\phi} - \P^i \one m ((\mean{i}{\phi} \H_i)\circ
\TF^i )]~dx|\le j^{\chi}
\]
where $n^{1+\chi}=o(\sigma_n^4)$ otherwise the double summation
contributes a term which is too large.

So we divide the sum into two parts, with $0 < \delta<1$
\[
\sum_{i=j-{j^{\delta}}}^{j-1} \int [(\mean{j}{\phi} \H_j)-m ((\mean{j}{\phi}
\H_j)\circ \TF^j)] \P_{i+1}^{j-i} [\P^i \one \H_i \mean{i}{\phi}- \P^i \one m
((\mean{i}{\phi} \H_i)\circ \TF^i)]~dx
\]
\[ + \sum_{i=1}^{j-j^{\delta}} \int [(\mean{j}{\phi} \H_j)-m ((\mean{j}{\phi}
\H_j)\circ \TF^j)] \P_{i+1}^{j-i} [\P^i \one \H_i \mean{i}{\phi}- \P^i \one m
((\mean{i}{\phi} \H_i)\circ \TF^i )]~dx.
\]
We bound the first sum by $C^* j^{\delta}$ using $L^2$ bounds without
decay. The second sum uses our decay estimate (see Lemma~\ref{lem:decay2})
and we get $\sum_{i=1}^{j-j^{\delta}} \frac{C^* i}{(j-i)^{\alpha^*}} \le
C^* j^{1-(\alpha^*-1)\delta}=C^*\ j^{1+\delta-\alpha^*\delta}$
provided $\alpha^*>1$ ($\iff 0 < \alpha < 1/2$). Then
$|\sum_{i=1}^{j-1} \int [(\mean{j}{\phi} \H_j)-m ((\mean{j}{\phi} \H_j)\circ \TF^j)]
\P_{i+1}^{j-i} [\P^i \one \H_i \mean{i}{\phi}- \P^i \one m ((\mean{i}{\phi} \H_i)\circ
\TF^i )]~dx|\le C(j^{\delta}+j^{1+\delta-\alpha^*\delta})$ which is
lowest for $\delta=1/\alpha_*$. We obtain
\begin{multline*}
  |\sum_{j=1}^n \sum_{i=1}^{j-1} \int [(\mean{j}{\phi} \H_j)\circ \TF^j -m
  ((\mean{j}{\phi} \H_j)\circ \TF^j)][(\mean{i}{\phi} \H_i)\circ \TF^i -m ((\mean{i}{\phi}
  \H_i)\circ \TF^i)]dx| \\ \le C^* \ n^{1+1/\alpha_*} = C^*\
  n^{1/(1-2\alpha)}
\end{multline*}
so
\begin{equation*}
  E(|S_n-E_n|^2) \le C n^{1/(1-2\alpha)}.
\end{equation*}
{\optbf
  By dividing for $\sigma_n^4$ and asking again for a growth like
  $\sigma_n^2\gtrsim  n^{\beta}$ we have now that
  $\beta>\frac{1}{2(1-2\a)}.$}
This estimate allows us to show that
$\frac{1}{\sigma_n^2}\left(\sum_{j=1}^n (\mean{j}{\phi} \H_j)\circ \TF^j -
  \sum_{j=1}^n E[ (\mean{j}{\phi} \H_j)\circ \TF^j]\right)\to 0$ in
probability.
%
\end{proof}

%

We now collect the various inequalities involving $\a$ and $\beta$, which
is the scaling of $\sigma_n^2\gtrsim  n^{\beta}$:
\begin{itemize}
\item for our proof of condition (ii) in Brown's Theorem~\ref{thm:Brown} we
  need $\a<\frac 1 5$ and $\beta>\frac{4\alpha}{1-\alpha}$, $\beta> \frac 1
  2$;

\item in Peligrad's argument we needed $\a< \frac 1 5$ and $\beta >
  \frac{4\a}{1-\a}$;

\item in Lemma \ref{lem:phi-H}, using that $\a< \frac 1 5$, we have $\beta
  > \frac{1}{2(1-2\a)}$.

\item for Theorem~\ref{bU} we use $\b > \frac 1 2$, and $\a< \frac 1 9$ to
  obtain a uniform $L^4$-bound for
  $\psi_n\circ \TF^n=\mean{n}{\phi}\circ \TF^{n} +\H_n \circ \TF^{n}
  -\H_{n+1} \circ \TF^{n+1}$ (see Corollary~\ref{cor:h-bounded}).
\end{itemize}
Therefore, it is sufficient to take
\begin{equation}\label{eq:alpha-beta-bounds}
  0 <\a < \alphaupperbound \text{ and } \b > \max\left\{\frac{1}{2},
    \frac{4\a}{1-\a}, \frac{1}{2(1-2\a)}\right\} =
  \frac{1}{2(1-2\a)}.
\end{equation}


To conclude the proof we need Theorem~\ref{bU} to show
that~\eqref{eq:condition_i_main-part} implies condition (i) of Brown's
Theorem~\ref{thm:Brown}, and the statement of Lemma \ref{lem:decay2}.

\begin{lemma}\label{lem:decay2}
  For $1\le p < 1/\a$
  \begin{equation}\nonumber
    \|\P_k^n\left([\P^i \one \H_i \mean{i}{\phi}- \P^i \one m ((\mean{i}{\phi} \H_i)\circ
      \TF^i )]\right)\|_p\le i\ C_{\alpha,p} \ C_{\phi} \
    n^{-\frac{1}{p\alpha}+1}
    \left(\log n\right)^{\frac{1}{\alpha}\frac{1-\a p}{p- \a p}}
\end{equation}
\end{lemma}

\begin{proof}
  See Section~\ref{sec.proof-in-appendix} in the Appendix.
\end{proof}

\begin{thm}\label{bU}
  Assume $\psi_j\circ \TF^j$ is uniformly bounded in $L^4$ and
  $\sigma_n^2 = \E(\sum_{j=1}^n \psi_j^2\circ \TF^j) + O(1) \gtrsim
  n^{\beta}$ with $\beta > \frac 1 2$. Then
  \[ \frac{1}{\sigma_n^2} \sum_{j=1}^n ( \psi^2_j \circ \TF^j- \E[ \psi^2_j
    \circ \TF^j |\mathcal{B}_{j+1}]) \to 0 \text{\quad in probability.}
  \]
\end{thm}

\begin{proof}
  Define
  \[
    V_k:=\psi^2_k \circ \TF^k- \E[ \psi^2_k \circ \TF^k
    |\mathcal{B}_{k+1}], \qquad T_n:=\sum_{j=1}^n V_j.
  \]
  Note that $\E[V_k | \mathcal{B}_{k+1}]=0$, so $V_k$ is a reverse
  martingale difference; by Pythagoras,
  $\E(V_k^2)\le \E((\psi^2_k \circ \TF^k)^2)$. Applying Pythagoras again,
%
%
%
%
%
%
   \[
     \E\Big[\big(\sum_{j=1}^n V_j\big)^2\Big] = \sum_{j=1}^n \E(V_j^2) \le
     \sum_{j=1}^n \E(\psi^4_j \circ \TF^j) \lesssim n
   \]
   therefore, by Tchebycheff
   \[
   P(|T_n|> \sigma_n^2 \epsilon) =P(|T_n|^2> \sigma_n^4
   \epsilon^2)\lesssim \frac{n}{\epsilon^2\sigma_n^4}
 \]
 Since we assumed that $\sigma_n^2\gg n^{1/2}$, it follows that
 \[
   \frac{1}{\sigma_n^2} T_n \to 0 \text{ in probability}
 \]
 as claimed.
\end{proof}

\section{Central Limit Theorem for nearby maps}\label{sec:nearby-maps}

\begin{thm}\label{thm:linear-variance}
  Given $\bbb \in (0, \alphaupperbound)$ and $\phi \in C^1([0,1])$, if
  $\phi$ is not a coboundary (up to a constant) for $T_\bbb$ there exists
  $\epsilon >0$ such that for all parameters
  $\beta_k\in (\bbb-\epsilon,\bbb+ \epsilon)$ the variance grows linearly
  for {\it any} sequential system formed from concatenation of the maps
  $T_{\beta_k}$.

  Therefore, by Theorem~\ref{thm:CLT},
  the CLT holds.
\end{thm}

\begin{proof}

  Recall the quantities defined by a concatenation of different maps:
  \begin{equation}\nonumber
    \H_n =\frac{1}{\P^n \one} \left[P_n
      (\mean{n-1}{\phi} \P^{n-1} \one) +P_n P_{n-1}
      (\mean{n-2}{\phi} \P^{n-2} 1)+ \dots + P_n P_{n-1} \dots P_1 (\mean{0}{\phi} \P^0
      \one)\right]
  \end{equation}
  and
  \[
  \psi_n :=\mean{n}{\phi} +\H_n-\H_{n+1} \circ T_{n+1}.
  \]

  First assume that the maps all coincide with $T_{\bbb}$ so that
  $P_\bbb^n 1\to h_{\bbb}$ (at a polynomial rate in $L^2$),
  $P_nP_{n-1}...P_{n-k}=P_{\bbb}^k$, where $h_{\bbb}$ is the
  invariant density for $T_{\bbb}$ and $P_{\bbb}$ is the transfer
  operator for $T_{\bbb}$ with respect to Lebesgue measure.
  Furthermore $\mean{n}{\phi}=\phi- m(\phi (T^n_{\bbb}))\to \phi-\int \phi
  h_{\bbb} dx$. Denote the $\H_n$ corresponding to this situation by
  $\H_{\bbb,n}$.

  Note the terms $P_n P_{n-1}...P_{n-j} (\mean{n-j-1}{\phi} \P^{n-j-1} \one)$
  decay at a polynomial rate in $L^2$, $\|P_n P_{n-1}...P_{n-j}
  (\mean{n-j-1}{\phi} \P^{n-j-1} \one)\|_2\le \frac{C}{j^{\tau}}$ for some
  $\tau>1$ for $\bbb < 1/4$, by Proposition~\ref{prop:decay_Lp} and
  Lemma \ref{lem:C1*Pk}. Note that $C$ and $\tau$ may be taken as
  uniform over all $T_{\beta_k}$ if $\beta_k$ is close to~$\bbb$.

  Combining this with the fact that $P_{\bbb}^n \one\to h_{\bbb}$ in
  $L^2$ (and hence $\frac{1}{P_{\bbb}^n \one} \to \frac{1}{h_{\bbb}}$
  in $L^2$ as both $h_{\bbb}$ and $P_{\bbb}^n \one$ are bounded below
  by a positive constant\footnote{These facts, in particular that $P_{\bbb}^n \one$ are uniformly in $n$ bounded from below by a strictly positive constant, are proved in \cite{LSV}. } ), we see that given $\epsilon >0$ there exists an $N$
  such that for all $n>N$, $\H_{\bbb,n}= \frac{1}{h_{\bbb}} [P_{\bbb}
  (h_{\bbb} \phi- \int \phi h_{\bbb} dx) +P_{\bbb}^2 (h_{\bbb} \phi-
  \int \phi h_{\bbb} dx)+ ...+ P_{\bbb}^N (h_{\bbb} \phi- \int \phi
  h_{\bbb} dx)] + \gamma (\bbb,n)$ where $\|\gamma (\bbb,n)\|_2 <
  \epsilon$.  We define $G_{\bbb,N}= \frac{1}{h_{\bbb}} [P_{\bbb}
  (h_{\bbb} \phi- \int \phi h_{\bbb} dx) +P_{\bbb}^2 (h_{\bbb} \phi-
  \int \phi h_{\bbb} dx)+ ...+ P_{\bbb}^N (h_{\bbb} \phi- \int \phi
  h_{\bbb} dx)]$ so that $\H_{\bbb,n}=G_{\bbb,N}+\gamma(\bbb,n)$.

  Now suppose $\phi$ is not a coboundary for $T_{\bbb}$. Denote by
  $\tilde{P}_{\bbb}$ the transfer operator for $T_{\bbb}$ with
  respect to the invariant measure $d\mu_{\bbb}=h_{\bbb} dx$. Then
  $\tilde{P}_{\bbb}^n (\phi)=\frac{1}{h_{\bbb}} P_{\bbb}^n
  (h_{\bbb} \phi)$ where $P_{\bbb}$ is the transfer operator for
  $T_{\bbb}$ with respect to Lebesgue measure.

  Hence $\frac{1}{h_{\bbb}} [P_{\bbb} (h_{\bbb} \phi- \int \phi
  h_{\bbb} dx) +P_{\bbb}^2 (h_{\bbb} \phi- \int \phi h_{\bbb}
  dx)+ ...+ P_{\bbb}^N (h_{\bbb} \phi- \int \phi h_{\bbb}
  dx)]=\sum_{k=1}^N \tilde{P}_{\bbb}^k [\phi-\int \phi
  d\mu_{\bbb}]$. If $\phi$ is not a coboundary then
  $\sum_{k=1}^{\infty} \tilde{P}_{\bbb}^k [\phi-\int \phi
  d\mu_{\bbb}]$ converges to a coboundary $\tilde{H}_{\bbb}$ so
  that
  \[
  \phi= \tilde{\psi}_{\bbb}+ \tilde{H}_{\bbb}\circ T_{\bbb}-\tilde{H}_{\bbb}
  \]
  defines a martingale difference sequence $\{
  \tilde{\psi}_{\bbb}\circ T_{\bbb}^n\} $, where
  $\tilde{\psi}_{\bbb}\not=0$ in $L^2$ (as $\phi$ is not a
  coboundary for $T_{\bbb}$). Suppose $\|\tilde{\psi}_{\bbb}\|_2>
  \eta$.

  Choose $N$ large enough that for all $n>N$, $\|
  [H_{\bbb,n}-H_{\bbb,n+1}\circ T_{\bbb}]-
  [\tilde{H}_{\bbb}-\tilde{H}_{\bbb}\circ T_{\bbb}]\|_2
  <\frac{\eta}{20}$ and $\|\tilde{H}_{\bbb}-\sum_{k=1}^N
  \tilde{P}_{\bbb}^k [\phi-\int \phi d\mu_{\bbb}]\|_2 <
  \frac{\eta}{20}$ . Then $\|\psi(\bbb,n)\|_2> \frac{\eta}{2}$ for
  all $n>N$.

  Now we consider a concatenation of maps $T_{\beta_k}$ where
  $\beta_k$ is close to $\bbb$.  The idea is to break $\H_n$ into a
  sum of $N$ terms uniformly close to $G(\bbb,N)$ (no matter what the
  sequence of maps) and a small error.

  Choose all $\beta_k$'s sufficiently close to $\bbb$ that when we
  form a concatenation of the maps $T_{\beta_k}$ we have
  \begin{multline*}
    \|G_{\bbb,N}- \frac{1}{\P^n \one} \big[P_n (\mean{n-1}{\phi} \P^{n-1}
    \one) +P_n P_{n-1} (\mean{n-2}{\phi} \P^{n-2} \one) + \dots \\ + P_n
    P_{n-1}\cdots P_{n-N} (\mean{n-N-1}{\phi} \P^{n-N-1} \one)\big]\|_2 <
    \frac{\eta}{20}.
  \end{multline*}
  We can do this as we have fixed $N$ and the finite terms are
  continuous in $L^2$ as $\beta_k \to \bbb$, see~\cite[Theorem
  5.1]{Leppanen-Stenlund} and Lemmas~\ref{lem:C1*Pk},
  \ref{lem:L2bounds}.

  Recall we also have $\|\gamma(\bbb,n)\|_2 < \frac{\eta}{20}$ for
  all $n\ge N$.

  Using the uniform contraction ($\tau$ and $C$ are uniform for
  $T_{\beta}$ where $\beta$ is in a small neighborhood of $\bbb$) we
  have
  \begin{multline*}
    \|\H_{n}- \frac{1}{\P^n \one} \big[P_n (\mean{n-1}{\phi} \P^{n-1} \one)
    +P_n P_{n-1} (\mean{n-2}{\phi} \P^{n-2} \one) + \dots \\ + P_n
    P_{n-1}\cdots P_{n-N} (\mean{n-N-1}{\phi} \P^{n-N-1}
    1)\big]\|_2<\frac{\eta}{20}
  \end{multline*}
  for all $n>N$.  Then $\|\psi_{n}\|_2> \frac{\eta}{10}$ for all $n>N$
  and we have linear growth of variance for the concatenation of maps
  as $\sigma_n^2=\sum_{k=1}^n E[\psi_n\circ\TF^k]^2$.
\end{proof}

\section{Random compositions of intermittent maps}\label{sec:quenched}

Suppose $S=\{ T_{\alpha_1}, \ldots, T_{\alpha_\ell } \}$ is a finite number
of intermittent type maps as in Section 1, with
$\alpha_i < \alphaupperbound$. We will take an iid selection of maps from
$S$ according to a probability vector $p=(p_1, \ldots, p_\ell)$ where the
probability of choosing map $T_{\alpha_i}$ is $p_i$. This induces a
Bernoulli measure $\nu$ on the shift space $\Omega:= \{ 1,\dots, l\}^{\N}$,
where $(i_1, i_2, \ldots, i_n, \ldots )$ corresponds to the sequence of
maps: first apply $T_{\alpha_{i_1}}$, then $T_{\alpha_{i_2}}$ and so on.
Writing elements of $\omega \in \Omega$ as sequences
$\omega:=(\omega_0,\omega_1,\ldots, \omega_n, \ldots)$ the shift operator
$S: \Omega \to \Omega$, $(S\omega)_{i}=\omega_{i+1}$ preserves the measure
$\nu$.

This random system also induces a Markov process on $[0,1]$ with the
transition probability function $P(x, A)=\sum_{i=1}^\ell p_{\alpha_i} 1_A
(T_{\alpha_i} (x))$. A measure $\mu$ is invariant for the Markov process if
$P^{*} \mu =\mu$. In this setting Bahsoun and Bose~\cite{Bahsoun_Bose} have
shown (among other results) that there is a unique absolutely continuous
invariant measure $\mu$ and that if $\phi: [0,1] \to \R$ is a H\"older
function then $\phi$ satisfies an {\it annealed} CLT for this random
dynamical system in the sense that if $\int \phi d~\mu=0$
then
\[
(\nu \times \mu) \{ (\omega, x): \frac{1}{n} \sum_{j=1}^n \phi (T_{(S^j\omega)_0} \ldots T_{(\omega)_0}  x) \in A\}
\rightarrow \frac{1}{{\sqrt{2\pi\sigma^2}}} \int_A e^{-\frac{x^2}{2\sigma^2}} dx
\]
for some $\sigma^2 \ge 0$.

In fact the result of Bahsoun and Bose~\cite{Bahsoun_Bose} also shows that
this convergence is with respect to $(\nu \times m)$ where $m$ is Lebesgue
measure on $[0,1]$. This follows from a well known result by Eagleson
\cite{EA} which states the equivalence of the convergence in distribution
for measures which are absolutely continuous with respect to each other.

We will strengthen this to a quenched result: almost every realization of
choices of concatenations of maps (with respect to the product measure
$\nu$), satisfies a self-norming CLT provided
$\phi$ is not a coboundary -- up to a constant -- for all maps (see the
precise statement below).
%
First we show that, in this situation, $\nu$ almost surely a random
composition of a finite number of intermittent type maps has linear growth
of the variance. Therefore, we can apply the CLT proven earlier.

\begin{lemma}
  Assume $\a_i < \frac{1}{4}$ for all $1\le i \le \ell$ and let
\[
\sigma_n^2 (\omega) : =\int \left(\sum_{j=1}^n \phi\circ \TF^j_\omega
-m(\phi\circ \TF^j_\omega)\right)^2~dx
\]
where $ \TF^j_\omega$ stands for $T_{(S^{j-1}\omega)_0} \circ \ldots \circ
T_{(\omega)_0}$.

If $ \phi$ is not a coboundary (up to a constant) for one of the maps, i.e.
\begin{quote}
  there exists an $i$ such that $\phi \not = c+\psi \circ T_{\alpha_i}
  -\psi$ for any measurable $\psi$ and any constant $c$
\end{quote}
then for $\nu$-almost every $\omega$ there exists a $C>0$ (independent of
$\omega$) and an integer $N(\omega)$ such that $\sigma_n^2(\omega) \ge Cn$
for all $n\ge N(\omega)$.
\end{lemma}
\begin{rmk}
  A similar cohomological condition was presented in~\cite{ALS} in the
  setting of two random commuting toral automorphisms, and conditions on
  the maps are given under which all $\phi\not=0$ have a linear rate of
  growth of variance.
\end{rmk}
\begin{proof}

  We will assume that $\phi$ is not a coboundary (up to constants) for one
  of the maps, suppose, without loss of generality, that this map is
  $T_{\alpha_1}$.

  Given any $k$, for $\nu$-a.e. $\omega$, the sequence of $m$ consecutive
  applications of the map $T_{\alpha_1}$ will occur in the sequence of
  composed maps prescribed by $\omega$ at a fixed asymptotic frequency of
  $p_1^m$.

  Now we consider $T_{\alpha_1}$ as a fixed map. $T_{\alpha_1}$ has an
  absolutely continuous invariant probability measure $\mu_{\alpha_1}$
  whose density~$h_{\alpha_1}$ is in the cone $\cone$. We let
  $Q_{\alpha_1}$ denote the transfer operator of $T_{\alpha_1}$ with
  respect to the invariant measure $\mu_{\alpha_1}$. Then
  $Q_{\alpha_1} \one=\one$, $P_{\a_1}h_{\alpha_1}=h_{\alpha_1}$.

  First we construct a martingale decomposition for $\phi$ using the
  transfer operator $Q_{\alpha_1}$ corresponding to the invariant measure
  $\mu_{\alpha_1}$ for $T_{\alpha_1}$. Note that $P_{\alpha_1}$ is the
  transfer operator of $T_{\alpha_1}$ with respect to Lebesgue measure $m$,
  so the relation between $P_{\alpha_1}$ and $Q_{\alpha_1}$ is
  $Q_{\alpha_1}(\phi)=\frac{1}{h_{\alpha_1}} P_{\alpha_1}(h_{\a_1}\phi)$,
  so
  $Q_{\alpha_1}^n (\phi)=\frac{1}{h_{\alpha_1}} P_{\alpha_1}^n
  (h_{\alpha_1} \phi)$ for all $n>0$. $Q_{\a_1}$ has the same decay rate as
  $P_{\a_1}$.

  Define
  \[
    H_{\alpha_1}=\sum_{j=1}^{\infty} Q_{\alpha_1}^j[\phi- \int \phi
    d\mu_{\alpha_1}] \qquad \text{and} \qquad \phi-\int \phi
    d\mu_{\alpha_1}=\psi_{\alpha_1} + H_{\alpha_1} -H_{\alpha_1}\circ
    T_{\alpha_1}.
  \]
  Although it will not be used, but note that
  $\{ \psi_{\alpha_1} \circ T_{\alpha_1}^n \}$ is a reverse martingale
  difference scheme with respect to $\mu_{\a_1}$ and the decreasing
  filtration $\mathcal{F}_n:=T_{\alpha_1}^{-n} \mathcal{B}$, where
  $\mathcal{B}$ is the $\sigma$-algebra of Borel sets on $[0,1]$.


  Since $\phi-\int \phi d \mu_{\a_1}=\psi_{\alpha_1} + H_{\alpha_1}
  -H_{\alpha_1}\circ T_{\alpha_1}$ and there are no measurable solutions to
  $\phi= c + H_{\alpha_1} -H_{\alpha_1}\circ T_{\alpha_1}$ with $c$
  constant, the martingale difference function $\psi_{\alpha_1}$ is not
  zero, so $\|\psi_{\alpha_1}\|_2>\rho>0$.\footnote{Unless explicitly
    stated, $L^2$ stands for $L^2(m)$, and conditional expectations are
    with respect to $m$.}

  Now we consider the analogous quantities defined by a concatenation of
  different maps, not just iterates of $T_{\alpha_1}$. We will use the
  notation from previous sections, so that
  $\P^n:=P_{\alpha_{i_n}}\circ P_{\alpha_{i_{n-1}}}\circ \cdots \circ
  P_{\alpha_{i_1}}$ for some sequence
  $\TF^n:=T_{\alpha_{i_n}}\circ T_{\alpha_{i_{n-1}}}\circ \cdots \circ
  T_{\alpha_{i_1}}$ (leaving out the dependence on $\omega$ for
  convenience).

  Defining as before
  \begin{equation}\nonumber 
    \H_n =\frac{1}{\P^n \one} \left[P_n (\mean{n-1}{\phi} \P^{n-1} \one) +P_n P_{n-1}
      (\mean{n-2}{\phi} \P^{n-2} \one)+ \dots + P_n P_{n-1} \dots P_1 (\mean{0}{\phi} \P^0
      \one)\right]
\end{equation}
and $\psi_n :=\mean{n}{\phi} +\H_n-\H_{n+1} \circ T_{n+1}$, the sequence
$\{ \psi_n\circ \TF^n\}$ is a reverse martingale difference scheme for $m$
and the decreasing filtration $\{\TF^{-n} \mathcal{B}\}$.

Our strategy is to show that if $k$ is sufficiently large (independent of
$n$) then $\|\psi_{n+2k}-\psi_{\alpha_1}\|_2<\frac{\rho}{2}$ and
$\|\psi_{n+2k+1}-\psi_{\alpha_1}\|_2<\frac{\rho}{2}$ every time that
$\psi_{n+2k}$ corresponds to the reverse martingale difference produced by
following any sequence of $n$ maps chosen from $S$ by $2k+1$ applications
of $T_{\alpha_1}$ (i.e., the last $2k+1$ maps applied were $T_{\a_1}$).

More precisely we will show that
$\|\H_{n+2k}-H_{\alpha_1}\|_2<\frac{\rho}{10}$ and
$\|\H_{n+2k+1}-H_{\alpha_1}\|_2<\frac{\rho}{10}$, which will imply
$\| \psi_{n+2k}-\psi_{\alpha_1}\|_2<\frac{\rho}{2}$ since
$\psi_{\alpha_1}-\psi_{n+2k}=[H_{\alpha_1}\circ
T_{\alpha_1}-H_{\alpha_1}]-[\H_{n+2k+1}\circ T_{\alpha_1}-\H_{n+2k}]$. The
proof that $\|\H_{n+2k+1}-H_{\alpha_1}\|_2<\frac{\rho}{10}$ is exactly the
same as the proof that $\|\H_{n+2k}-H_{\alpha_1}\|_2<\frac{\rho}{10}$, so
we only give details in the latter case. In fact, to simplify notation we
consider $2k$ applications of the $T_{\alpha_1}$ after $n$ applications of
any sequence of maps from $S$.

Once we have established this, by Lemma~\ref{lem:c},
\[
  \sigma_m^2 \approx \sum_{j=1}^{m}\E[\psi_j^2\circ \TF^j]=
  \sum_{j=1}^{m}\int\psi_j^2\cdot \P^j(\one) d m,
\]
and hence (since $\P^j(\one)$ is bounded away from zero) there is linear
growth as for any integer $r$, $r$ consecutive applications of
$T_{\alpha_1}$ will occur with an asymptotic frequency of $p_1^{r}$ for
$\nu$ a.e. $\omega$.

To set the stage for our estimates we make the assumption that $n$ maps
have been applied followed by $2k$ applications of $T_{\alpha_1}$ and write
\[
\H_{n+2k} =\frac{1}{\P^{n+2k} \one} [P_{n+2k} (\mean{n+2k-1}{\phi} \P^{n+2k-1} \one) +P_{n+2k} P_{n+2k -1}
  (\mean{n+2k -2}{\phi} \P^{n+2k-2} \one)
  \]
  \[
  + \dots + P_{n+2k} P_{n+2k-1} \dots P_1 (\mean{0}{\phi} \P^0 \one)]
  \]
as
\[
  H_{n+2k}=A(k,n)[B(k,n)+C(k,n)]
\]
where $A(k,n):=\frac{1}{\P^{n+2k} \one}$,
$B(k,n):=\sum_{j=0}^k P_{n+2k} P_{n+2k -1}\ldots P_{n+2k-j} (\mean{n+2k
  -j-1}{\phi} \P^{n+2k-j-1} \one)$ and
$C(k,n):=\sum_{j=k+1}^{n+2k-1}P_{n+2k} P_{n+2k -1}\ldots P_{n+2k-j}
(\mean{n+2k -j-1}{\phi} \P^{n+2k-j-1} \one)$.

Recall that
$\sum_{j=1}^k Q_{\alpha_1}^j [\phi-\int \phi d\mu_{\alpha_1}] =
\frac{1}{h_{\alpha_1}} [P_{\alpha_1} (h_{\alpha_1} \phi- h_{\alpha_1}\int
\phi h_{\alpha_1} dx) +P_{\alpha_1}^2 (h_{\alpha_1} \phi- h_{\alpha_1}\int
\phi h_{\alpha_1} dx)+ ...+ P_{\alpha_1}^k (h_{\alpha_1} \phi-
h_{\alpha_1}\int \phi h_{\alpha_1} dx)]$
converges to $H_{\alpha_1}$ at a polynomial rate in $L^2$.

We define $\alpha(k):=\frac{1}{h_{\alpha_1}}$ (which does not actually
depend on $k$),
$\beta(k):=P_{\alpha_1} (h_{\alpha_1} \phi- h_{\alpha_1}\int \phi
h_{\alpha_1} dx) +P_{\alpha_1}^2 (h_{\alpha_1} \phi- h_{\alpha_1}\int \phi
h_{\alpha_1} dx)+ ...+ P_{\alpha_1}^{k+1} (h_{\alpha_1} \phi-
h_{\alpha_1}\int \phi h_{\alpha_1} dx)$ and
$\gamma (k):=\sum_{j=k+2}^{\infty} P_{\alpha_1}^j
[h_{\alpha_1}\phi-h_{\alpha_1}\int \phi d\mu_{\alpha_1}]$.

We will show that as $k$ increases, uniformly in $n$,
$\|A(k,n)-\alpha(k)\|_2 \to 0$, $\|B(k,n)-\beta (k)\|_2 \to 0$,
$\|C(k,n)\|_2 \to 0$ and $\|\gamma (k)\|_2\to 0$. As
$H_{\alpha_1}= \alpha(k)[\beta (k) + \gamma (k)]$ and
$\H_{n+2k}=A(k,n)[B(k,n)+C(k,n)]$ this implies (because $A(k,n)$ and
$\a(k)$ are uniformly bounded in $L^\infty$) that
$\|\H_{n+2k}-H_{\alpha_1}\|_2\le \frac{\rho}{10}$ for sufficiently large
$k$.

We first consider the terms $A(k,n)$ and $\alpha(k)$.


For any $j$ and any sequence of $j$ maps
$T_{\alpha_{i_j}}\circ T_{\alpha_{_{j-1}}}\circ \cdots \circ
T_{\alpha_{i_1}}$ chosen from $S$ the corresponding transfer operator with
respect to Lebesgue measure
$\P^j=P_{\alpha_{i_j}}\circ P_{\alpha_{i_{j-1}}}\circ \cdots \circ
P_{\alpha_{i_1}}$ (again, we leave out the dependence on $\omega$ for
notational convenience) has the property that $ \P^j \one$ lies in the cone
$\cone$ and $\int \P^j \one dx=1$.

Furthermore for any $n$
\[
  P_{\alpha_1}^{2k} [h_{\alpha_1} -\P^n \one] \to 0 \text{ as } k\to\infty
\]
in $L^2$ at a uniform polynomial rate, in fact
$\|P_{\alpha_1}^{2k} [h_{\alpha_1} -\P^n \one] \|_2 \le
C\frac{1}{(2k)^{1+\eta}}$ where $C$ and $\eta$ are uniform over $\P^n 1$.

Hence $\frac{1}{P_{\alpha_1}^{2k}\P^n \one } \to \frac{1}{h_{\alpha_1}}$ in
$L^2$ at a polynomial rate as both $h_{\alpha_1}$ and
$P_{\alpha_1}^{2k}\P^n \one$ are uniformly bounded below by a positive
constant. Thus there exists $C_1>0$ such that for all $k$ and $n$
\[
  \| \frac{1}{P_{\alpha_1}^{2k}\P^n \one } - \frac{1}{h_{\alpha_1}}\|_2 \le
  C_1\frac{1}{(2k)^{1+\eta}}
\]
This is the same as
\[
  \| A(k,n) - \alpha(k)\|_2 \le C_1\frac{1}{(2k)^{1+\eta}}
\]

Now we consider $C(k,n)$ and $\gamma(k)$.

The terms $P_n P_{n-1}...P_{n-j}
(\mean{n-j-1}{\phi} \P^{n-j-1} \one)$ decay at a polynomial rate in $L^2$, in fact
$\|P_n P_{n-1}...P_{n-j}
(\mean{n-j-1}{\phi} \P^{n-j-1} \one)\|_2\le \frac{C}{j^{1+\eta}}$. Note that $C$ and $\eta$ may be taken as uniform over all choices of $T_{\alpha_i}$ in the concatenation.
Hence

\[
  \|C(k,n)\|_2 \le \frac{C_2}{(2k)^{\delta}}
\]
Similarly
\[
  \|\gamma (k) \|_2 \le \frac{C_3}{(2k)^{\delta}}
\]

Finally we consider the terms $B(k,n)$ and $\beta(k)$. Observe that
\[
  B(k,n)-\beta(k)=
  \sum_{j=0}^k P_{\a_1}^{j+1} \left[(\mean{n+2k -j-1}{\phi} \P^{n+2k-j-1}
    \one) -
    (h_{\alpha_1} \phi- h_{\alpha_1}\int \phi h_{\alpha_1} dx) \right]
\]
where the terms in square brackets have Lebesgue integral zero and are
``uniformly'' differences of functions in the cone $\cone$ (see
Corollary~\ref{cor:C1*Pk-applications}). Therefore
\[
   \left\|P_{\a_1}^{j+1} \left[(\mean{n+2k -j-1}{\phi} \P^{n+2k-j-1}
    \one) -
    (h_{\alpha_1} \phi- h_{\alpha_1}\int \phi h_{\alpha_1} dx)
  \right]\right\|_2 \le \frac{C}{j^{1+\delta}}.
\]
uniformly over $n$, $k$ and $j$.

Hence uniformly over $n$,
\[
  \|B(k,n)-\beta(k)]\|_2 \le \frac{C_2}{k^{1+\delta}}
\]

To summarize: we have shown that if we choose $k$ large enough then
$\| \H_{n+2k}-H_{\alpha_1}\|_2 <\frac{\rho}{10}$ and
$\| \H_{n+2k_1}-H_{\alpha_1}\|_2 <\frac{\rho}{10}$, hence
$\|\psi_{n+2k}\|_2>\frac{\rho}{2}$, whenever, independently of $n$, the
last $2k+1$ maps in the sequence are all $T_{\a_1}$.
This implies linear growth in the random composition setting as almost all
choices of maps will have $2k+1$ long sequences of the map $T_{\alpha_1}$
at a fixed frequency $p_1^{2k+1}$.
\end{proof}


The next theorem is an immediate consequence of the previous Lemma and
Theorem~\ref{thm:CLT}.

\begin{thm}\label{thm:CLTquenched}

  If $\alpha_i < \alphaupperbound$ for all $1\le i \le \ell$ and $ \phi$ is
  not a coboundary (up to constants) for one of the maps $T_{\a_i}$ then
  $\sigma_n^2(\omega)\ge C n$ for some $C>0$ and $n>N(\omega)$, and hence
  $\phi$ satisfies a CLT, for $\nu$ almost every sequence $\omega$ of maps.

\end{thm}

\section{Appendices}

\subsection{{\Sprindzuk} Theorem.}\label{sec.galkoksma-appendix}

We recall the following result, as formulated by W. Schmidt~\cite{W1,W2}
and stated by Sprindzuk~\cite{Sprindzuk}\footnote{Quoting
  Sprindzuk~\cite{Sprindzuk}: ``The Lemma is abstracted from the work of
  W.~Schmidt, and is based on the idea of the well-known method of
  Rademacher in the theory of orthogonal series.''}:


\begin{thm}[\protect{\cite[page 45, Lemma 10]{Sprindzuk}}]\label{thm:sprindzuk}
    Let $(\Omega,\mathcal{B},\mu)$ be a probability space and let $f_k
    (\omega) $, $(k=1,2,\ldots )$ be a sequence of non-negative $\mu$
    measurable functions and $g_k$, $h_k$ be sequences of real numbers such
    that $0\le g_k \le h_k \le 1$, $(k=1,2, \ldots,)$. Suppose there exists
    $C>0$ such that
  \begin{equation} \nonumber 
    \int \left(\sum_{m<k\le n}( f_k (\omega) - g_k)
    \right)^2\,d\mu \le C \sum_{m<k \le n} h_k
  \end{equation}
  for arbitrary integers $m <n$. Then for any $\epsilon>0$
  \[
  \sum_{1\le k \le n} f_k (\omega) =\sum_{1\le k\le n} g_k + O
  (\Theta^{1/2} (n) \log^{3/2+\epsilon} \Theta (n))
  \]
  \text{for $\mu$-a.e. $\omega \in \Omega$,} where
  $\Theta (n)=\sum_{1\le k \le n} h_k$.
\end{thm}

\subsection{\protect{Proof of Lemma
    \ref{lem:decay2}}}\label{sec.proof-in-appendix}

\begin{proof}
  For simplicity of notation we discuss only the case $k=1$; the
  general case is the same, since we use the $n$ Perron-Frobenius maps
  in $\P_k^n$ only for the decay given by Theorem~\ref{thm:decay}.

  The idea is to write $[\P^i \one \H_i \mean{i}{\phi}- \P^i \one m ((\mean{i}{\phi}
  \H_i)\circ\TF^i )]$ as a difference of $2i$ functions in the cone of
  the same integral. By writing explicitely $\H_i$ we get
 \[
 [\P^i \one \H_i \mean{i}{\phi}- \P^i \one m ((\mean{i}{\phi} \H_i)\circ\TF^i)] =
 \left[\sum_{k=1}^i\prod_{j=0}^{k-1}P_{i-j}(\mean{i-k}{\phi}\P^{i-k}\one )\mean{i}{\phi}-\P^i
   \one m((\mean{i}{\phi} \H_i)\circ\TF^i)\right]
 \]
 \[
 =\left[\sum_{k=1}^i\prod_{j=0}^{k-1}P_{i-j}(\mean{i-k}{\phi}\P^{i-k}\one)\mean{i}{\phi}-\P^i
   \one \sum_{k=1}^im((\mean{i}{\phi} \frac{1}{\P^i \one
   }\prod_{j=0}^{k-1}P_{i-j}(\mean{i-k}{\phi} \P^{i-1}\one )\circ
   \TF^i)\right]
 \]
 \[
 =\sum_{k=1}^i\left[\mean{i}{\phi} \P^k_{i-k+1}(\mean{i-k}{\phi}\P^{i-k}\one )-\P^i \one
   m((\mean{i}{\phi} \frac{1}{\P^i \one }\P^{k}_{i-k+1}(\mean{i-k}{\phi} \P^{i-1}\one)\circ
   \TF^i)\right]
\]
Call $C_{k,i}:=m((\mean{i}{\phi} \frac{1}{\P^i \one }\P^{k}_{i-k+1}(\mean{i-k}{\phi}
\P^{i-1}1)\circ\TF^i)$; then consider the quantity
 $$
 (*):=\mean{i}{\phi} \P^k_{i-k+1}(\mean{i-k}{\phi}\P^{i-k} \one)-\P^i \one C_{k,i}.
 $$
 Since $\mean{i-k}{\phi}\in C^1$ and $\P^{i-k}\one\in \cone$ we can write by
 Lemma~\ref{lem:C1*Pk}
$$
\mean{i-k}{\phi}\P^{i-k}\one=F_{i-k}-G_{i-k}
$$
with $F_{i-k}, G_{i-k}\in \cone.$ By the invariance of the cone, the functions $h^{(1)}_{i-k}:=\P^k_{i-k+1}F_{i-k}; \ h^{(2)}_{i-k}:=\P^k_{i-k+1}G_{i-k}$ are still in the cone, and we rewrite (*) as
$$
(*)= \mean{i}{\phi} h^{(1)}_{i-k}-\mean{i}{\phi} h^{(2)}_{i-k}-C_{i,k}\P^i \one.
$$
Although the functions (in the cone), $F_{i-k}, G_{i-k}$ are not of
zero mean, we can still apply Lemma~\ref{lem:C1*Pk} and split the
product of $\mean{i}{\phi}$ with them into the differences of two new
functions belonging to the cone, namely
$$
\mean{i}{\phi} h^{(1)}_{i-k}=M^{(1)}_{i-k}-M^{(2)}_{i-k}; \ \mean{i}{\phi}
h^{(2)}_{i-k}=N^{(1)}_{i-k}-N^{(2)}_{i-k}
$$
with $M^{(1,2)}_{i-k}, N^{(1,2)}_{i-k}\in \cone.$ We finally have
$$
(*)=[M^{(1)}_{i-k}+N^{(2)}_{i-k}]-[M^{(2)}_{i-k}+N^{(1)}_{i-k}+C_{i,k}\P^i \one ]:=R_{i,k}-S_{i,k}
$$
where the functions $R_{i,k}, S_{i,k}$ are in the cone and have the
same expectation. Before continuing, let us summarize what we got
$$
[\P^i \one \H_i \mean{i}{\phi}- \P^i \one m ((\mean{i}{\phi} \H_i)\circ\TF^i
)]=\sum_{k=1}^i (R_{i,k}-S_{i,k}).
$$
By taking the power $\P^n$ on both sides we have by our
Theorem~\ref{thm:decay} on the loss of memory and
Proposition~\ref{prop:decay_Lp}
$$
\|\P^n\left([\P^i \one \H_i \mean{i}{\phi}- \P^i \one m ((\mean{i}{\phi}
  \H_i)\circ\TF^i )]\right)\|_p\le \sum_{k=1}^i
C_{\alpha,p}(\|R_{i,k}\|_1+\|S_{i,k}\|_1)n^{-\frac{1}{p\alpha}+1}
\left(\log n\right)^{\frac{1}{\alpha}\frac{1-\a p}{p- \a p}}.
$$
From Lemma~\ref{lem:C1*Pk}, one observes that if we have $\phi\in
C^1([0,1])$ and $H\in \cone$ the splitting $\phi H=A-B$, with $A,B\in
\cone$ is such that the functions $A,B$ depend only on the $C^1$ norm
of $\phi$ and the integrals $m(H), m(\phi H).$ In our case since
$\mean{i}{\phi}(x)=\phi(x)-m(\phi\circ \TF^i)$, we have that
$\|\mean{i}{\phi}\|_{C^1}\le \|\phi\|_{C^1};$ moreover, at each application of
Lemma~\ref{lem:C1*Pk}, the function $H$ is either $\P^i \one $ or
obtained by applying $\P^\ell$ to a function obtained in the previous
step and which only depends upon $\|\phi\|_{C^1};$ in conclusion the
norms $\|R_{i,k}\|_1, \|S_{i,k}\|_1$ are bounded by a function
$C_{\phi}$ which only depends on the choice of the observable $\phi.$
We finally get
 $$
 \|\P^n\left(\P^i \one [\H_i \mean{i}{\phi}- m((\mean{i}{\phi} \H_i)\circ \TF^i)]
 \right)\|_p \le i\ C_{\alpha,p} \ C_{\phi} \ n^{-\frac{1}{p\alpha}+1}
 \left(\log n\right)^{\frac{1}{\alpha}\frac{1-\a p}{p- \a p}}.
 $$
\end{proof}

\section*{Acknowledgments}
SV was supported by the ANR-Project {\em Perturbations} and by the PICS
(Projet International de Coop\'eration Scientifique), {\em Propri\'et\'es
  statistiques des syst\`emes dynamiques d\'eterministes et al\'eatoires},
with the University of Houston, n. PICS05968. SV thanks the University of
Houston for supporting his visits during the preparation of this work. SV
thanks the Leverhulme Trust for support thorough the Network Grant
IN-2014-021, and for useful conversations with S. Galatolo and W. Bahsoun.
MN was supported by NSF grant DMS 1101315 and Simons Foundation
Collaboration Grant Number 349664. AT was partially supported by the Simons
Foundation grant 239583. The authors warmly thank R. Aimino for useful
suggestions concerning Theorem~\ref{thm:SBC}. All three authors are
grateful for the support received from the Erwin Schr\"odinger Institute in
Vienna during their May 2016 visit.

\end{document}